\documentclass[a4paper,fleqn,reqno,11pt]{amsart}

\usepackage[T1]{fontenc}
\usepackage[english]{babel}
\usepackage[utf8]{inputenc}

\usepackage[top=2cm,bottom=3cm,left=2cm,right=2cm]{geometry}

\usepackage{amsmath}
\usepackage{amsfonts}
\usepackage{amssymb}
\usepackage{amsthm}
\usepackage{enumitem}
\usepackage{mathtools}
\mathtoolsset{showonlyrefs}
\usepackage{graphicx}

\usepackage{textcomp}
\usepackage{color}

\usepackage{hyperref}
\hypersetup{
	colorlinks=true, 
	urlcolor     = blue, 
	linkcolor=blue,  
	citecolor=blue   
}

\usepackage{mathrsfs}

\newtheorem{theorem}{Theorem}[section]
\newtheorem{lemma}[theorem]{Lemma}
\newtheorem{proposition}[theorem]{Proposition}

\newtheorem{remark}[theorem]{Remark}
\newtheorem{definition}[theorem]{Definition}

\newcommand*{\medcap}{\mathbin{\scalebox{1.3}{\ensuremath{\cap}}}}

\newcommand\bs[1]{{\boldsymbol #1}}

\def\P{\mathbb{P}}
\def\Q{\mathbb{Q}}
\def\R{\mathbb{R}}
\def\N{\mathbb{N}}

\def\one{{\mathbf 1}}

\allowdisplaybreaks
\numberwithin{equation}{section}

\title[Fluid limit for condensing zero-range processes]{Fluid limit for the coarsening phase \\ of the condensing zero-range process}
\author[I. Armend\'ariz, J. Beltr\'an, D. Cuesta and M. Jara]
{In\'es Armend\'ariz, Johel Beltr\'an, Daniela Cuesta and Milton Jara}

\address{In\'es Armend\'ariz
\hfill\break\indent
Universidad de Buenos Aires \& IMAS-CONICET-UBA, 
\hfill\break\indent
Buenos Aires, Argentina.}
\email{{\tt iarmend@dm.uba.ar}}

\address{Johel Beltr\'an
\hfill\break\indent
Pontificia Universidad Cat\'olica del Per\'u \& IMCA,
\hfill\break\indent
Lima, Per\'u.}
\email{{\tt johel.beltran@pucp.edu.pe}}

\address{Daniela Cuesta
\hfill\break\indent
Universidad de Buenos Aires \& IMAS-CONICET-UBA \& Universidad Torcuato Di Tella,
\hfill\break\indent
Buenos Aires, Argentina.}
\email{{\tt dcuesta@utdt.edu}}

\address{Milton Jara
\hfill\break\indent
Instituto de Matem\'atica Pura e Aplicada,
\hfill\break\indent
Rio de Janeiro, Brasil.}
\email{{\tt monets@impa.br}}

\begin{document}
\sloppy	

\begin{abstract} 
\noindent 
We prove a fluid limit for the coarsening phase of the condensing zero-range process on a finite number of sites. When time and occupation per site are linearly rescaled by the total number of particles, the evolution of the process is described by a piecewise linear trajectory in the simplex indexed by the sites. The linear coefficients are determined by the trace process of the underlying random walk on the subset of non-empty sites, and the trajectory reaches an absorbing configuration in finite time. A boundary of the simplex is called absorbing for the fluid limit if a trajectory started at a configuration in the boundary remains in it for all times. We identify the set of absorbing configurations and characterize the absorbing boundaries.   \\

\end{abstract}

\maketitle

\section{Introduction}\label{intro}

The zero-range process was introduced in \cite{Spitzer} as an interacting particle system on a graph, where the jump rates of particles depend only on the occupation of the departure site. Once a particle jumps, it changes its position according to a transition probability $p$.  When the jump rates are decreasing, particles spend more time in highly occupied sites and the dynamics favours the formation of clusters. As the density diverges, in a typical stationary configuration, a macroscopic number of particles accumulates on a single site, the condensate. 
During the past twenty years, this phenomenon has been rigorously studied for the families of condensing zero-range processes with decreasing rates introduced in \cite{ Drouffe_1998, Evans_2000}, and several results including the equivalence of ensembles, a law of large numbers
and a central limit theorem for the condensate size have been established in \cite{MR1797308, MR2013129, AL, JBCL, MR3071386}.

The next step is to study the time evolution of the process, which involves several time scales. The slowest time scale is related to the dynamics of the condensate. When the graph $V$ is finite and the number of particles tends to infinity, starting from a stationary initial configuration with a condensate, the position of the condensate follows a Markov chain on $V$ with jump rates that are proportional to the capacities between pairs of sites of the random walk with transition probabilities $p$. This metastable description was first derived for reversible, supercritical systems  \cite{JBCL}, extended to totally asymmetric dynamics \cite{Landim_asymm}, general non-reversible dynamics \cite{MR3922538}, and recently to critical, reversible zero-range processes \cite{LMS_critical}. The analogous problem in the thermodynamical limit, when the number of sites in $V$ increases together with the number of the particles, was studied in \cite{AGL}.

Slower times scales are related to the coarsening phase that leads to a stationary state with a condensate at a single site. Starting from an initial uniform distribution, sites exchange particles until only a few accumulate the excess number of particles; these are the cluster sites. This transition occurs on a fast hydrodynamic time scale. On a slower time scale, cluster sites interact, with some clusters growing (macroscopically) at the expense of the others, up to the time when only one of them remains. The different stages of this description were first predicted in \cite{MR1987434, MR2165701}, which were followed by several studies in the theoretical physics literature \cite{MR1987434, MR2165701, MR2314354, MR3491315, MR3584414}. 

Beltr\'an, Jara and Landim \cite{BJL} consider the second phase of the coarsening transition for zero-range processes on a fixed graph $V$. They show that on a  time scale of order $N^2$, $N$ the diverging number of particles, the vector of linearly rescaled cluster sizes converges to an absorbed diffusion in the simplex
$\Sigma=\{ \sum_{i\in V} u_i=1,\, u_i\ge 0,\, i\in V\}$. The vector is absorbed upon hitting a boundary, and thereafter, it evolves as a diffusion on a lower dimensional simplex. The process stops when the diffusion reaches a vertex of $\Sigma$, the macroscopic manifestation of the condensed phase. These results confirm the previous description, and identify the precise dynamics of the transition.

In this paper we establish the fluid dynamics of the initial phase of the coarsening process for a broad family of zero-range processes on a fixed graph, including condensing zero-range processes \cite{ Drouffe_1998, Evans_2000}. If time and occupation per site are linearly rescaled by the total number of particles $N$, the process approximately follows an explicit piecewise linear trajectory in $\Sigma$, with velocities determined by the rates of the random walk underlying the zero-range dynamics. The trajectory evolves in progressively lower dimensional simplices until it eventually stabilizes when it arrives at an absorbing configuration. 

Boundaries of the simplex are characterized as absorbing or non-absorbing.
Inspired by the characterization of stochastic processes by their martingale properties, we identify the piecewise linear trajectory as the unique solution to a generalized ODE that instantaneously exits non-absorbing boundaries. For initial configurations on absorbing boundaries, the ODE formulation alone suffices to characterize the trajectory, and the problem reduces to proving that weak limits satisfy the equation. To derive the result for the general case, we show that the rescaled zero-range process exits non-absorbing boundaries using a coupling construction that compares the process with an open queueing network.

It is a simple observation that when the jump rates out of non-empty sites are constant, the zero-range process reduces to a closed Jackson network on the graph. Chen and Mandelbaum derive a fluid limit theorem for this process when time is scaled linearly in the number of particles \cite{chen-mandelbaum-fluid}, and show that under diffusive scaling it converges to a Brownian motion reflected at the boundaries of the simplex  \cite{chen-mandelbaum-diffusion}. The first result is a particular case of the fluid limit derived in this paper: closed Jackson networks belong to the class of zero-range processes covered by our results. On the other hand, we see that the macroscopic evolutions of Jackson networks and condensing zero-range processes already diverge on a diffusive time scale, as the latter converge to diffusions absorbed at the boundaries \cite{BJL}.

Over the years, there have been several studies of the hydrodynamic limit of the zero-range process describing the macroscopic evolution of the empirical density as the number of sites tends to infinity, in different regimes \cite{KL,MR4089493}. For condensing zero-range processes this is a difficult problem, as the standard approach does not apply, and moreover the expected limiting partial differential equation is not always well-posed; partial results are established in \cite{MR3296275,Loula-Stama}.

The article is organized as follows.

In \S~\ref{Namr} we introduce notation and state a series of results. We provide an explicit construction of the trajectory of the fluid limit  
and state Theorem~\ref{c2}, our main result, in \S~\ref{FLimit}. 

In \S~\ref{charac} we propose a new characterization of the fluid limit as a weak solution (in some precisely defined sense) to an ordinary differential equation that instantaneously exits non-absorbing sets, Definition \ref{propab}, and identify the solution to the ORP described in \S~\ref{orp0} as the unique solution to this problem. 

We study the family of rescaled zero-range processes in \S~\ref{s2}. 
We prove tightness in \S~\ref{tight}, and show that weak limits of the zero-range process are supported on solutions of the ODE, \S~\ref{mgle-prblm}. In order to completely characterize fluid limits, it remains to prove that they exit non-absorbing boundaries. We first establish this property when the zero-range process reduces to a system of queues (constant rates) and extend it to the general case by a coupling argument in \S~\ref{coupling}. We collect all the information to conclude the proof of our main result in \S~\ref{dem-c2}.

\section{Setup and results}\label{Namr}
\subsection{The zero-range process}\label{zrp}	

Throughout this article we denote $\N_0=\{0,1,2,\dots\}$ and $\R_+=[0,\infty)$. Let us fix a finite set $V$ and consider the zero-range process on $V$ with jump rates
\begin{align}
g:\N_0\rightarrow \R_+, \quad \text{such that} \;g(0)=0,
\end{align}
and underlying random walk determined by an \emph{irreducible} set of transition rates
\begin{align}\label{rates}
r=\big(r(i,j)\big)_{i,j\in V},\quad \text{so that} \quad r(i,i)=0, \; \forall i\in V.
\end{align}
In this process, particles are indistinguishable and a particle leaves a given site $i$ at rate $r(i)g(n)$, where $n$ is the number of particles at $i$ and $r(i)$ is given in \eqref{probr} below. Once a particle jumps from $i$, it moves to a site $j$ chosen with probability given by
\begin{align}\label{probr}
p(i,j):=\frac{r(i,j)}{r(i)}, \quad \text{where} \quad r(i):=\sum_{j\in V} r(i,j).
\end{align}
Formally, the zero-range process $(\eta_s,\,s\ge 0)$ is a Markov process with state space $\N_0^V$ and generator acting on each $f:\N_0^V\rightarrow\R$ by
	\begin{align}\label{gen-zrp}
	Lf(\eta)=\sum_{i,j\in V}g(\eta(i))\,r(i,j)\,[f(\eta^{i,j})-f(\eta)],\quad \eta\in \N_0^V,
	\end{align}
with
\begin{equation}\label{eta-ij}
\eta^{i,j}(l):=\left\{
\begin{array}{ll}
\eta(i)-1,& \text{ for } l=i,\\
\eta(j)+1,& \text{ for } l=j,\\
\eta(l),& \text{ for } l\neq i,j.
\end{array}\right.
\end{equation}
Since the number of particles is preserved by these transitions, when the zero-range process $(\eta_s,\,s\ge 0)$ starts with $N$ particles, the rescaled process $(\zeta^N_t,\,t\ge 0)$ defined as
\begin{align}\label{zeta}
\zeta^N_t=\dfrac{\eta_{tN}}{N},\quad t\ge 0,
\end{align}
can be viewed as a Markov process with state space the standard simplex on the set of vertices $V$:
\begin{align}\label{simplex}
\Sigma=\big\{u \in [0,1]^{V}:\sum\limits_{i\in V}u(i)=1\big\}.
\end{align}
We aim to study the sequence $(\zeta^N_t,\,t\ge 0)$, $N\ge 1$, under the assumption that $g(n)$ converges to a strictly positive constant. Since such constant can be absorbed by the transition rates $r$, we may and will assume that
\begin{align}\label{rate-f}
\lim_{n \to \infty} g(n)=1.
\end{align}
We will also assume that 
\begin{align}\label{rate-fbis}
g(n)\ge 1,\,n\in \N.
\end{align}
This assumption will only be used in the coupling construction in Section \ref{coupling}, to minimize technical details, but the arguments can be adapted without much difficulty to the general case.

Denote by $D(\R_+,\Sigma)$ the space  of $\Sigma$-valued right continuous functions with left limits endowed with the Skorokhod topology and the respective borel sigma field. In addition, let $C(\R_+, \Sigma)$ stand for the space of $\Sigma$-valued continuous trajectories. Consider the following sequence of discretizations of $\Sigma$:
\begin{align}
\Sigma_{N}=\big\{u\in\Sigma:Nu\in\N_0^V\big\}, \quad N\in \N.
\end{align}
For each $N\in \N$ and $u\in \Sigma_N$, let $\P^N_u$ be the law on $D(\R_+,\Sigma)$ of the rescaled process $(\zeta^N_t,\,t\ge 0)$, when the zero-range process $(\eta_s,\,s\ge 0)$ starts at $Nu\in\N_0^V$.

\begin{proposition}\label{tight2}
Under hypothesis \eqref{rate-f}, the family of probability laws $\big( \P^N_u\big)_{N\in \N, u\in \Sigma_N}$ is sequentially compact. In addition, any limit point is supported on $C(\R_+, \Sigma)$.
\end{proposition}

Actually, we shall prove in our main result that the limit points of $( \P^N_u )$ are deterministic continuous paths on $\Sigma$. The description of such paths makes use of the trace of the underlying random walk in a fundamental fashion.

\subsection{Trace process}\label{tracep}

Let us recall the notion of the trace process. Let $(X^k_s,\,s\ge 0)$ stand for a continuous-time random walk on $V$, with transition rates $r$ and starting at $k\in V$. For each nonempty $B\subseteq V$, let us denote by $T^k_B$ the hitting time of $B$ starting at $k$:
\begin{align}\label{hitting}
T^k_B=\inf\{s\ge 0,\, X^k_s\in B\}. \quad \text{In addition, denote $T^k_j=T^k_{\{j\}}$, for $j\in V$}.
\end{align}
Since $r$ is irreducible, then $T^k_B$ is almost surely finite. Now, we define $r^A=(r^A(i,j))_{i,j\,\in A}$, for each nonempty $A\subseteq V$, as $r^A(i,i)=0$, $\forall i\in A$, and
\begin{align}\label{trace-rates}
r^A(i,j)=r(i,j) + \sum_{k\notin A} r(i,k)\,{\rm P}(T^k_A=T^k_j),\quad \forall i,\,j\in A, i\not = j.
\end{align}
Trivially, $r^V=r$. The set $r^A$ corresponds to the rates of an irreducible random walk on $A$, called the trace process. A more detailed description and properties of the trace process can be found in \cite{BL}, Section 6.1. From now on, we will refer to $r^A$ as \emph{the trace of $r$ on $A$}. 

We may easily compute $r^A$, recursively. When $A=V\setminus \{k\}$, equation \eqref{trace-rates} reduces to
\begin{align}\label{ator}
r^A(i,j)=r(i,j) + r(i,k)\,p(k,j),\quad \forall i,\,j\in A, i\not = j,
\end{align}
where $p(k,j)$ is given by \eqref{probr}. Since the trace of $r^A$ on any nonempty $B\subseteq A$ coincides with $r^B$ (see e.g. \cite{BL}) then \eqref{ator} can be used repeatedly to compute the trace of $r$ on any nonempty subset of $V$. 

\subsection{Minimal $r$-absorbing sets}

In this subsection we introduce a terminology for a class of subsets of $V$ and state some of its properties. For each nonempty $A\subseteq V$, and $i\in A$, let us define
\begin{align}\label{lambdaA}
\lambda^A(i):=\sum_{j\in A} [r^A(j,i)-r^A(i,j)]. \quad \text{Observe that} \; \sum_{i\in A} \lambda^A(i) = 0.
\end{align}

\begin{definition}\label{rabsor}
Let us say that a subset $A$ of $V$ is $r$-absorbing if $A=V$ or
\begin{align}
\lambda^{A\cup \{j\}}(j)\le 0, \quad \forall j\in V\setminus A.
\end{align}
\end{definition}
Notice that $\varnothing$ is trivially $r$-absorbing. We will be interested in the minimal $r$-absorbing set containing a given $S\subseteq V$. For this purpose, the following lemma provides a key result.

\begin{lemma}\label{a37}
Let $B$ be $r$-absorbing and let $A$ be a proper subset of $B$. If $\lambda^B(i)\le 0$, $\forall i\in B\setminus A$, then $A$ is also $r$-absorbing.
\end{lemma}

As an immediate consequence of Lemma \ref{a37} we have the following result.

\begin{lemma}\label{lproma}
Given $S\subseteq V$ there exists an $r$-absorbing subset $A$ such that
\begin{align}\label{proma}
S \subseteq A \quad \text{and} \quad \lambda^A(i)>0, \; \forall i\in A\setminus S.
\end{align}
\end{lemma}
\begin{proof}
We provide an algorithm to generate such subset. Initialize with any $r$-absorbing $A_1$ containing $S$, for instance $A_1=V$. Compute
\begin{align}
O_1:=\{ i \in A_1\setminus S : \lambda^{A_1}(i)\le 0\}.
\end{align}
If $O_1= \varnothing$ then $A_1$ is the desired subset. Otherwise, set $A_2 := A_1\setminus O_1$ and clearly $A_2\subsetneq A_1$. By Lemma \ref{a37}, $A_2$ is $r$-absorbing. Since $S\subseteq A_2$, we may repeat the procedure. This process must stop since $V$ is finite and we get at the end the desired subset.
\end{proof}

We will show that the class of $r$-absorbing subsets is stable by intersection.

\begin{lemma}\label{abinter}
If $A$ and $B$ are $r$-absorbing then $A\cap B$ is $r$-absorbing.
\end{lemma}

This property permits us to define the minimal $r$-absorbing subset containing a given $S$.

\begin{definition}\label{def-min-abs}
For each $S\subseteq V$, we define $\mathcal A(S)$ as the intersection of all $r$-absorbing subsets of $V$ containing $S$.
\end{definition}

It follows from Lemma \ref{a37} that $\mathcal A(S)$ must satisfy \eqref{proma}. Actually, it will be very important to notice that property \eqref{proma} characterizes it.

\begin{lemma}\label{min-abs}
For any nonempty $S \subseteq V$, $\mathcal A(S)$ is the unique $r$-absorbing subset of $V$ satisfying \eqref{proma}.
\end{lemma}

Therefore, the proof of Lemma \ref{lproma} provides an algorithm to compute $\mathcal A(S)$.


\subsection{The fluid limit}\label{FLimit}

In this subsection we construct the set of continuous paths 
\begin{align}
\zeta^u\in C(\R_+,\Sigma), \quad u\in \Sigma,
\end{align}
that arises as the limit points of $(\P^N_u)$. Let us start introducing some ingredients. For each $u\in\Sigma$, denote
\begin{align}\label{ssu}
\mathcal S(u):=\{ i \in V : u(i)>0 \} \quad {and} \quad \mathcal A(u) := \mathcal A\big(\mathcal S(u)\big).
\end{align}
Also, define the set of vectors $\lambda^u\in \R^V$, $u\in \Sigma$ as
\begin{align}
\lambda^u(i) :=
\left\{
\begin{array}{cl}
\lambda^{\mathcal A(u)}(i), & \text{for $i\in \mathcal A(u)$,} \\
0, &  \text{otherwise}.
\end{array}
\right.
\end{align}
Let $T^u$ be the exit time from $\Sigma$ of a path starting at $u$ and with constant velocity $\lambda^{u}$:
\begin{align}
T^u := \min \{ t\ge 0 : u + t \lambda^{u} \not \in \Sigma \}, \quad \text{where} \quad \min \varnothing = \infty.
\end{align}
By Lemma \ref{min-abs}, $\lambda^{u}(j)> 0$, $\forall j\in \mathcal A(u)\setminus \mathcal S(u)$. Hence, in order to compute $T^u$, it suffices to consider the set of coordinates
\begin{align}
\mathcal S_{\downarrow}(u) := \{ i\in \mathcal S(u) : \lambda^{u}(i)<0 \}.
\end{align}
Since $\sum_{i\in \mathcal A(u)} \lambda^{u}(i)=0$ then
\begin{align}\label{tinfi}
\mathcal S_{\downarrow}(u)=\varnothing \; \iff \; \lambda^{u} \;\text{is null} \; \iff \; T^u=\infty.
\end{align}
Furthermore, if $\mathcal S_{\downarrow}(u)$ is nonempty then
\begin{align}\label{mintu}
0 < \min \left\{ - \frac{u(i)}{\lambda^{u}(i)} : i\in \mathcal S_{\downarrow}(u) \right\} = T^u < \infty.
\end{align}

\begin{lemma}\label{zuend}
Given $u\in \Sigma$, if $T^u<\infty$ then
\begin{align}
\mathcal S(v) = \mathcal A(v) \subsetneq \mathcal A(u), \quad \text{where} \quad v:=u + T^u\lambda^u.
\end{align}
\end{lemma}
\begin{proof}
Since $\lambda^u$ vanishes outside $\mathcal A(u)$, then $\mathcal S(v) \subseteq \mathcal A(u)$. Now, for any $j\in \mathcal A(u)\setminus \mathcal S(v)$ we have
\begin{align}
0 = v(j) = u(j) + T^u \lambda^{u}(j) \quad \implies \quad \lambda^{\mathcal A(u)}(j) = -\frac{u(j)}{T^u} \le 0. \quad \text{(Recall that $T^u>0$.)}
\end{align}
Since $\mathcal A(u)$ is $r$-absorbing, we may apply Lemma \ref{a37} to conclude that $\mathcal S(v)$ is $r$-absorbing. On the other hand, since $T^u<\infty$, then $S_{\downarrow}(u)$ is nonempty. Thus, in virtue of \eqref{mintu}, by taking any
\begin{align}
{\boldsymbol i} \in \text{argmin}\left\{ - \frac{u(i)}{\lambda^{u}(i)} : i \in S_{\downarrow}(u) \right\}
\end{align}
we have $v({\boldsymbol i})=0$, that is ${\boldsymbol i}\not \in \mathcal S(v)$. Since ${\boldsymbol i}\in S_{\downarrow}(u)\subseteq \mathcal A(u)$, we are done.
\end{proof}

Now, for each $u\in \Sigma$, we compute a finite sequence of pairs
\begin{align}\label{fsop}
(T_n, v_n)\in [0,\infty)\times\Sigma , \quad n=0,1,\dots,f,
\end{align}
as follows. Initialize with $(T_0,v_0)=(0,u)$. Suppose $(T_n,v_n)$ has already been defined for $n=0,1,\dots k$. If $\lambda^{v_k}$ is null, we stop and $k$ equals $f$ in \eqref{fsop}. Otherwise, we compute
\begin{align}
\mathcal S_{\downarrow}(v_k)=\{ i\in \mathcal S(v_k) : \lambda^{v_k}(i)<0 \} \quad \text{and} \quad T^{v_k}=\min \left\{ -\frac{v_k(i)}{\lambda^{v_k}(i)} : i\in S_{\downarrow}(v_k) \right\}.
\end{align}
By \eqref{tinfi} and \eqref{mintu}, we have $0<T^{v_k}<\infty$ and we define $(v_{k+1}, T_{k+1})$ as
\begin{align}
v_{k+1}:=v_{k} + T^{v_{k}} \lambda^{v_k} \quad \text{and} \quad T_{k+1}:=T_k+T^{v_k}.
\end{align}
We now repeat the same procedure. By Lemma \ref{zuend}, the cardinality of $\mathcal A(v_{n})$ is strictly decreasing and therefore, this process must stop. 

We finally construct $\zeta^u$ as a concatenation of the rectilinear paths determined by $(T_n,v_n)$:
\begin{align}\label{zetaudef}
\zeta^u_t = \left\{
\begin{array}{cl}
v_k + (t- T_k)\lambda^{v_k}, & \text{if $T_k\le t < T_{k+1}$ and $k < f$}, \\
v_f,& \text{for all $t\ge T_f$}.
\end{array}
\right.
\end{align}
It is not difficult to verify from its construction that this family of paths satisfies the following property:
\begin{align}\label{fpr}
\text{For all $s,t\ge 0$ and $u\in \Sigma$, we have $\zeta^{v}_t = \zeta^{u}_{s+t}$, where $v=\zeta^u_s$}.
\end{align}
We may now state our main result which establishes a fluid limit for the zero-range process.

\begin{theorem}\label{c2}
Let $u_N\in \Sigma_N$, $N\ge 1$, be a sequence converging to some $u\in \Sigma$. Under assumptions \eqref{rate-f}-\eqref{rate-fbis}, $\P^N_{u_N}$ converges weakly to the Dirac measure $\delta_{\zeta^u}$, where $\zeta^u\in C(\R_+,\Sigma)$ is defined in \eqref{zetaudef}.
\end{theorem}

Hence, when time and occupation levels are rescaled linearly, the discrete, stochastic flow of individual particles converges to a deterministic limit that evolves in the continuum. When the jump process rates are constant, $g(n)=\one_{\{n\ge 1\}}$, the zero-range process reduces to a closed Jackson network on the set $V$, and the result was first derived by Chen and Mandelbaum in \cite{chen-mandelbaum-fluid}.

\subsection{Absorbing faces and bottlenecks} Let us remark some important features of the fluid limit. For each $A\subseteq V$, let us denote by $\partial_A\Sigma$ the $A$-face of $\Sigma$:
\begin{align}
\partial_A \Sigma :=\{\zeta\in\Sigma: \sum_{j\in A} \zeta(j)=1 \}.
\end{align}
We start justifying our terminology of $r$-absorbing sets. 

\begin{remark}\label{rpab}
If $A$ is $r$-absorbing then $\partial_A\Sigma$ turns out to be an absorbing face in the following sense: for any $u\in \Sigma$ and $s\ge 0$,
\begin{align}\label{pab}
\zeta^u_s\in \partial_A \Sigma \quad \implies \quad \zeta^u_t\in \partial_A \Sigma,\,\forall t\ge s.
\end{align}
\end{remark}
Indeed, by Lemma \ref{zuend}, we have
\begin{align}\label{obs1}
\mathcal S(\zeta^u_t)=\mathcal A(\zeta^u_t)\subseteq \mathcal A(\zeta^u_s), \quad \text{whenever $0\le s < t$.}
\end{align}
In particular,
\begin{align}\label{abst}
\mathcal S(\zeta^u_t) \; \text{is $r$-absorbing}, \quad \forall t>0.
\end{align}
Then, $\zeta^u_s\in \partial_A\Sigma$ implies that $\mathcal A(\zeta^u_s)\subseteq A$ and \eqref{pab} follows from \eqref{obs1}. 

Notice now that, even if $B$ is not $r$-absorbing, it could contain an $r$-absorbing set $A$. In that case, if $\mathcal S(u)\subseteq A$ then $\zeta^u_t\in \partial_B\Sigma$, $\forall t\ge 0$. Nevertheless, $\lambda^{\mathcal A(B)}(j)>0$, for all $j\in \mathcal A(B)\setminus B$, and so
\begin{remark}
If $B$ is not $r$-absorbing and $\mathcal S(u)=B$ then $\zeta^u$ exits from $\partial_B \Sigma$ instantly:
\begin{align}
\inf\{ t> 0 : \zeta^u_t\not\in \partial_B\Sigma \} = 0.
\end{align}
\end{remark}
We finish by providing a characterization of those faces where the fluid limit equilibrates. Let $\mu(i)$, $i\in V$ be the invariant distribution associated to $r$, i.e.
\begin{align}
\sum_{i\in V} \big( \mu(i) r(i,j) - \mu(j)r(j,i) \big) = 0, \quad \forall i\in V.
\end{align}
Let $\mathcal M$ stand for the set of sites with maximal $\mu$-measure:
\begin{align}
\mathcal M:=\{i\in V,\, \mu(i)=\max_{j\in V}\mu(j)\}.
\end{align}
\begin{proposition}\label{absconf}
Let $A$ be a nonempty subset of $V$. The following assertions are equivalent.
\begin{itemize}
\item[$a)$] $u\in \partial_A\Sigma$ $\implies$ $\zeta^u_t=u$, $\forall t\ge 0$.
\item[$b)$] $A$ is $r$-absorbing and $\lambda^A$ is null.
\item[$c)$] $A\subseteq \mathcal M$.
\end{itemize}
\end{proposition}

In the terminology of Jackson networks, sites in $\mathcal M$ are called \emph{bottlenecks}. The fact that the fluid limit of closed networks equilibrates at a finite time, after which all nonbottlenecks remain empty, was already proved by \cite{chen-mandelbaum-fluid}, see also \S 7.10 in \cite{chen-yao}.

\subsection{Oblique reflection problem}\label{orp0}

In this subsection, we relate the fluid limit with the so-called oblique reflection problem. Let $p=\big( p(i,j)\big)_{i,j\in V}$ be an \emph{irreducible} set of transition probabilities so that $p(i,i)=0$, for all $i\in V$. For each $i\in V$, let $1_{\{i\}}\in \R^V$ stand for the indicator of $i$. Define the set of vectors
\begin{align}\label{dor}
{\boldsymbol w}_i := \sum_{j\in V} p(i,j)\big( 1_{\{i \}} - 1_{\{j \}} \big), \quad i\in V.
\end{align}
Denote by $D(\R_+,\R^V)$ the space of $\R^V$-valued right continuous functions with left limits and let $\R_+^V$ stand for the nonnegative orthant $\{ u \in \R^V : u(i)\ge 0, \;\forall i\in V \}$.
\begin{definition}\label{c5}
Given $\xi\in D(\R_+,\R^V)$, a pair $\zeta, \rho\in D(\R_+,\R^V)$ is said to solve the oblique reflection problem (ORP) if
\begin{itemize}
\item[i)] For each $i\in V$, the path $\rho_{t}(i)$, $t\ge 0$, is nondecreasing and
\begin{align}\label{zdr}
\rho_t(i) = \int_{0}^{t} 1_{\{\zeta_s(i)=0\}} d \rho_s(i), \quad \forall t\ge 0.
\end{align}

\item[ii)] For all $t\ge 0$, we have $\zeta_t = \xi_t + \sum_{i\in V} \rho_t(i) {\boldsymbol w}_i$ and $\zeta_t \in \R_+^V$.

\end{itemize}
\end{definition}
Roughly speaking, solving the ORP amounts to finding a constrained version $\zeta$ of an input path $\xi$ that is restricted to live in $\R_+^{V}$, so that, when $\zeta$ hits a face $\{u\in \R^V_{+} : u(i)=0 \}$, threatening to go negative, increments of $\rho(i)$, the $i$-th coordinate of the so-called regulator $\rho$, push $\zeta$ along the direction ${\boldsymbol w}_i$. Condition \eqref{zdr} restricts each coordinate $\rho(i)$ to increase only at times $t\ge 0$ when $\zeta_t(i)=0$. A very interesting interpretation of the solution $(\zeta,\rho)$ as a temporal evolution of a Leontief economy can be found in \cite{chen-mandelbaum}.

Notice that the ORP is determined by the irreducible set of transition probabilities $p$. In addition, let us fix a vector $\lambda\in \R^V$ satisfying
\begin{align}\label{hl}
\sum_{i\in V}\lambda(i)=0.
\end{align}
The following result is an immediate consequence of Theorem 2.5 in \cite{chen-mandelbaum} and we state it here for future reference.

\begin{proposition}\label{cymp}
For each $u\in \Sigma$, there exists a unique $(\zeta^{\lambda,u},\rho^{\lambda,u})$ solving the ORP with input
\begin{align}
\xi^u_t = u + t\lambda, \quad t\ge 0.
\end{align}
Furthermore, $\zeta^{\lambda,u}\in C(\R_+,\Sigma)$ and the map $\Gamma:\Sigma \to C(\R_+,\Sigma)$ given by $\Gamma(u)=\zeta^{\lambda,u}$, is continuous when $C(\R_+,\Sigma)$ is endowed with the topology of uniform convergence on compact subsets of $\R_+$.
\end{proposition}

We claim that each $\zeta^{\lambda,u}$, $u\in\Sigma$, coincides with the fluid limit $\zeta^{u}$ for a suitably chosen $r$.

\begin{proposition}\label{porp}
If $r=(r(i,j))_{i,j\in V}$ is related to $p$ and $\lambda$ by \eqref{probr} and 
\begin{align}\label{lww}
\lambda = -\sum_{i\in V} r(i){\boldsymbol w}_i,
\end{align}
then $\zeta^{\lambda,u} = \zeta^u$, for all $u\in \Sigma$, where $\zeta^u$ is the path constructed in \eqref{zetaudef} using $r$.
\end{proposition}

We complement the previous proposition assuring that it is always possible to choose such $r$.

\begin{lemma}\label{lorp}
Given an irreducible set of transition probabilities $p$, satisfying $p(i,i)=0$, $\forall i\in S$ and a vector $\lambda$ satisfying \eqref{hl}, there exists some $r(i)>0$, $i\in V$, such that \eqref{lww} holds.
\end{lemma}

\section{Characterization of the fluid limit}\label{charac}

In this section we propose a more suitable alternative to the ORP introduced in Subsection \ref{orp0}. In this problem (see Definition \ref{propab}), we make use of a family of test functions to characterize the behaviour of a path. This feature fits with the martingale approach we shall use to prove Theorem \ref{c2}. Additionally, this alternative problem provides a natural connection between the paths constructed in \eqref{zetaudef} and the solutions of the ORP. 

Let us start introducing the space of test functions. Denote
\begin{align*}
C^1(\Sigma, \R):=\big\{f\in C^1(U,\R) : U\text{ open in }\R^V, \Sigma\subset U \subseteq \R^V\big\}.
\end{align*}
For each $f\in C^1(\Sigma, \R)$ and $v\in \R^V$, we define $v(f):\Sigma \to \R$ as the derivative of $f$ in direction $v$:
\begin{align}
v(f)(u)= \frac{\partial f}{\partial v} (u), \quad u\in \Sigma.
\end{align}
Recall that $r$ is an irreducible set of transition rates on $V$. It determines the set of vectors
\begin{align}\label{bsi}
{\boldsymbol v}_i = \sum_{k\in V} r(i,k)\big(1_{\{k\}} - 1_{\{i\}}\big), \quad i\in V.
\end{align}

\begin{definition}
For each $i\in V$, let ${\mathcal D}_i$ denote the family of functions $f\in C^1(\Sigma, \R)$ satisfying
\begin{align}\label{s24}
{\boldsymbol v}_i(f)(u)=0, \quad \text{for all $u\in \Sigma$ such that $u(i)=0$}.
\end{align}
In addition, for every nonempty $B\subseteq V$, let us denote ${\mathcal D}_B:=\medcap_{i\in B} {\mathcal D}_i$. 
\end{definition}
The set of transition rates $r$ also determines the vector
\begin{align}\label{ladef}
\lambda := \sum_{i\in V} {\boldsymbol v}_i, \quad \text{or equivalently} \quad \lambda(j) = \sum_{i\in V}\big[ r(i,j) - r(j,i) \big], \; j\in V. 
\end{align}
Observe that $\lambda$ coincides with $\lambda^V$ in \eqref{lambdaA} and also satisfies \eqref{lww} if $p$ corresponds to $r$ as in \eqref{probr}. Recall the notion of $r$-absorbing subsets of $V$ introduced in Definition \ref{rabsor}. Let us denote,
\begin{align}\label{siabs}
\Sigma_{abs} := \{ u\in \Sigma : \;\text{$\mathcal S(u)$ is $r$-absorbing}\}
\end{align}
where $\mathcal S(u)$ is the support of $u$, as defined in \eqref{ssu}.

\begin{definition}\label{propab}
Let us say that $\zeta\in C(\R_+,\Sigma)$ solves the $(\lambda,\mathcal D_V)$-problem if it satisfies the following two conditions.
\begin{itemize}
\item[\bf (A)] For every $f\in \mathcal D_V$,
\begin{align}
f(\zeta_t) = f(\zeta_0) + \int_0^t \lambda(f)(\zeta_s)\,ds, \quad \forall t\ge 0.
\end{align}
\item[\bf (B)] The path $\zeta$ enters instantly to $\Sigma_{abs}$, i.e. $\inf\{ t \ge 0 : \zeta_t \in \Sigma_{abs} \} = 0.$
\end{itemize}
\end{definition}

The main goal of this section is to prove that, for each $u\in \Sigma$, the path $\zeta^{u,\lambda}$ introduced in Proposition \ref{cymp} is the unique solution of the $(\lambda,\mathcal D_V)$-problem starting at $u$.

\subsection{Trace and harmonic extension}

For each nonempty $A\subseteq V$ and $v,w\in \R^A$ we denote
\begin{align}
\langle v, w \rangle = \sum_{i\in A} v(i)w(i).
\end{align}
The Markov generator $L:\R^V\to \R^V$ associated to $r$ is given by 
\begin{align}
Lw(i) = \langle \bs v_i , w\rangle, \quad w\in \R^V, i\in V,
\end{align}
where $\bs v_i$ is the vector given in \eqref{bsi}. Let us say that $w\in \R^V$ is harmonic on $B\subseteq V$ if
\begin{align}
\langle \bs v_k , w\rangle=0, \; \forall k\in B.
\end{align}
Recall that $(X^k_s,\,s\ge 0)$ stands for a continuous-time random walk with transition rates $r$ and starting at $k\in V$, and we denote by $T^k_B$ the respective hitting time of $B\subseteq V$, as defined in \eqref{hitting}. For each nonempty $A\subseteq V$ and $i\in A$, define the vector $1^A_i\in \R^V$ as
\begin{align}
1^A_i(k) = P(T^k_A = T^k_i), \quad k\in V.
\end{align}
It is well known that $1^A_i$ is the unique vector satisfying
\begin{align}
\left\{
\begin{array}{l}
1^A_i=1_{\{i\}}, \;\text{on $A$}, \\
1^A_i \; \text{is harmonic on $V\setminus A$}.
\end{array}
\right.
\end{align}
Given $w\in \R^A$, let us denote by $w^A\in \R^V$ its (unique) harmonic extension:
\begin{align}\label{haex}
\left\{
\begin{array}{l}
w^A=w, \;\text{on $A$}, \\
w^A \; \text{is harmonic on $V\setminus A$}.
\end{array}
\right.
\end{align}
It is easy to verify that
\begin{align}\label{hext}
w^A = \sum_{i\in A} w(i)1^A_i.
\end{align}
Recall from Subsection \ref{tracep} the notion of the trace of $r$ on $A$, denoted by $r^A$. For each nonempty $A\subseteq V$, define the set of vectors $\bs v^A_i\in \R^A$, $i\in A$ as 
\begin{align}
\bs v^A_i = \sum_{k\in A} r^A(i,k)\big(1_{\{k\}} - 1_{\{i\}}\big).
\end{align}
The following lemma provides an alternative definition of vectors $\bs v^A_i$.

\begin{lemma}\label{viaw}
For any nonempty $A\subseteq V$, $w\in \R^A$ and $i\in A$ we have
\begin{align}
\langle \bs v^A_i , w\rangle = \langle \bs v_i , w^A \rangle
\end{align}
where $w^A\in \R^V$ stands for the harmonic extension of $w$, as in \eqref{haex}.
\end{lemma}

\begin{proof}
For each $i,j\in A$, $i\not = j$ we have
\begin{align}\label{pref}
\langle \bs v^A_i , 1_{\{j\}} \rangle = r^A(i,j) = r(i,j) + \sum_{k\not \in A} r(i,k) 1^A_j(k) = \sum_{k\in V} r(i,k) 1^A_j(k) = \langle \bs v_i , 1^A_{j} \rangle.
\end{align} 
Since $\bs v^A_i$ is orthogonal to any constant vector in $\R^A$ and similarly to $\bs v_i$, it follows from \eqref{pref} that
\begin{align}
\langle \bs v^A_i , 1_{\{i\}} \rangle = - \sum_{j\in A \setminus \{i\}} \langle \bs v^A_i , 1_{\{j\}} \rangle = - \sum_{j\in A \setminus \{i\}} \langle \bs v_i , 1^A_{j} \rangle = \langle \bs v_i , 1^A_{i} \rangle.
\end{align}
Finally, the claim follows from \eqref{hext}.
\end{proof}

\subsection{Proof of Lemma \ref{a37}}\label{ssa37}

Recall the definition of the vector $\lambda^A\in \R^A$ from \eqref{lambdaA} or equivalently,
\begin{align}
\lambda^A = \sum_{i\in A} {\bs v}^A_i, \quad \text{for each nonempty $A\subseteq V$}.
\end{align}
It follows from Lemma \ref{viaw} that, for any nonempty $A\subseteq V$ and $w\in \R^A$,
\begin{align}\label{liaw}
\langle \lambda^A , w\rangle = \langle \lambda , w^A \rangle, \quad\text{where $w^A\in \R^V$ stands for the harmonic extension of $w$.}
\end{align}
The following lemma allows us to relate the vectors $\lambda^A$, $A\subseteq V$.

\begin{lemma}\label{lambdaab}
Given $A\subsetneq B\subseteq V$, with $A$ nonempty, we have
\begin{align}
\lambda^A(j) = \sum_{k\in B} \lambda^B(k)1^A_j(k), \quad \forall j\in A.
\end{align} 
\end{lemma}

\begin{proof}
By using Lemma \ref{viaw} we have
\begin{align}\label{jj1}
\lambda^A(j) = \sum_{i\in A} \langle  {\bs v}^A_i , 1_{\{j\}} \rangle = \sum_{i\in A}\langle  {\bs v}_i , 1^A_{j} \rangle = \sum_{i\in B}\langle  {\bs v}_i , 1^A_{j} \rangle
\end{align}
In the last equality we used that $1^A_j$ is harmonic on $B\setminus A$. Now, since $1^A_j$ is also the harmonic extension of $1^A_j|_B$ (restriction of $1^A_j$ to $B$), then Lemma \ref{viaw} provides
\begin{align}\label{jj2}
\langle {\bs v}^B_i, 1^A_{j}|_B\rangle = \langle  {\bs v}_i , 1^A_{j} \rangle, \quad \text{for each $i\in B$}
\end{align}
Putting \eqref{jj1} and \eqref{jj2} together, we get the desired result.
\end{proof}

\begin{proof}[Proof of Lemma \ref{a37}]
We assume that $B$ is $r$-absorbing, $A$ is a proper subset of $B$ and
\begin{align}\label{hipo1}
\lambda^B(k)\le 0, \quad \forall k\in B\setminus A.
\end{align}
Fix first an arbitrary $\ell \in B\setminus A$. Applying Lemma \ref{lambdaab} for $\lambda^{A\cup \{\ell\}}$ and $\lambda^{B}$, we get
\begin{align}
\lambda^{A\cup \{\ell\}}(\ell) = \sum_{k\in B} \lambda^B(k)1^{A\cup \{\ell\}}_{\ell} (k) = \sum_{k\in B\setminus A} \lambda^B(k)1^{A\cup \{\ell\}}_{\ell} (k).
\end{align}
This last expression is nonpositive due to \eqref{hipo1}. Fix now an arbitrary $j\in V\setminus B$. Applying Lemma \ref{lambdaab} for $\lambda^{A\cup \{j\}}$ and $\lambda^{B\cup \{j\}}$ we get
\begin{align}\label{jj3}
\lambda^{A\cup\{j\}}(j) = \lambda^{B\cup\{j\}}(j) + \sum_{k\in B\setminus A} \lambda^{B\cup \{j\}}(k)1^{A\cup\{j\}}_j(k) \le \lambda^{B\cup\{j\}}(j) + \sum_{k\in B\setminus A} \lambda^{B\cup \{j\}}(k).
\end{align}
Let us now apply Lemma \ref{lambdaab} to relate $\lambda^{B}$ and $\lambda^{B\cup \{j\}}$:
\begin{align}\label{jj4}
\lambda^B(k) = \lambda^{B\cup\{j\}}(j) 1^{B}_{k}(j) + \lambda^{B\cup\{j\}}(k), \quad \text{for each $k\in B$}.
\end{align}
Recall that $1^{B}_{k}(j)=P( T^j_B = T^j_k )$. Then, by summing up all the terms in \eqref{jj4}, for $k\in B\setminus A$, we get
\begin{align}\label{jj5}
\sum_{k\in B\setminus A} \lambda^B(k) = \lambda^{B\cup\{j\}}(j) P( T^j_{B} = T^j_{B\setminus A}) +  \sum_{k\in B\setminus A} \lambda^{B\cup\{j\}}(k).
\end{align}
By using \eqref{jj5} in \eqref{jj3} and hypothesis \eqref{hipo1} we obtain
\begin{align}
\lambda^{A\cup\{j\}}(j) \le \lambda^{B\cup\{j\}}(j) P\big( T^j_{B} = T^j_{A}\big)  + \sum_{k\in B\setminus A} \lambda^B(k) \le \lambda^{B\cup\{j\}}(j).
\end{align}
Since $B$ is $r$-absorbing then $\lambda^{A\cup\{j\}}(j) \le 0$. We have thus shown that $A$ is $r$-absorbing.
\end{proof}


\subsection{Absorption}\label{absb}

The main feature of any $\zeta\in C(\R_+,\Sigma)$ satisfying {\bf (A)} is the emergence of absorbing faces in the sense of Remark \ref{rpab}. In this subsection, we prove a preliminary result in this direction. Namely, in Lemma \ref{absort}, we prove that if $\zeta\in C(\R_+, \Sigma)$ satisfies {\bf (A)} and $S=\mathcal S(\zeta_0)$ is $r$-absorbing then $\zeta$ does not exit $\partial_{S}\Sigma$ as long as its $S$-coordinates are strictly positive. 

Given $h\in {\mathcal D}_B$, it will be useful to get a function $f\in {\mathcal D}_V$, so that $f\equiv h$ and $\lambda(f)\equiv \lambda(h)$ on a large subset of $\Sigma$. This is guaranteed by the following lemma whose proof is very technical and we postpone to Section \ref{perturb}.

\begin{lemma}\label{extens}
Let $A\subseteq V$ be nonempty and $B=V\setminus A$. Given $h\in {\mathcal D}_B$ and $\epsilon>0$, there exists $f\in {\mathcal D}_V$ such that
\begin{align}\label{a21}
u\in \Sigma,\,\min_{i\in A} u(i) \ge \epsilon \quad \implies \quad f(u)=h(u)\quad\text{ and }\quad \lambda(f) (u)= \lambda(h)(u).
\end{align}
\end{lemma}

In the following lemma we show how $\lambda^A$ arises as the velocity of a certain projection of $\zeta$ on $\R^A$.

\begin{lemma}\label{linear}
If $\zeta\in C(\R_+,\Sigma)$ satisfies {\bf (A)} then, for any nonempty $A\subseteq V$ and $0\le s_1 < s_2$,
\begin{align}
\underset{s_1\le t \le s_2}{\min_{i\in A}} \zeta_t(i) >0 \quad \implies \quad \langle 1^A_i,\zeta_t \rangle = \langle 1^A_i,\zeta_{s_1}\rangle + \lambda^A(i)(t-s_1), \;\forall i\in A, \forall t\in [s_1,s_2].
\end{align}
\end{lemma}

\begin{proof} 
Fix some $i\in A$ and consider $h(u) = \langle 1^A_i,u \rangle$, $u\in \Sigma$. Since, for all $u\in \Sigma$, we have
\begin{align}
{\boldsymbol v_j}(h)(u) = \langle {\boldsymbol v_j} , 1^A_i \rangle, \quad \forall j\in V,
\end{align}
then $h\in \mathcal D_{V\setminus A}$. Due to \eqref{liaw} we have
\begin{align}
\lambda(h)(u) = \langle \lambda , 1^A_i \rangle = \lambda^A(i), \quad \forall u\in \Sigma.
\end{align}
Denote $\epsilon:= \min\{ \zeta_t(j) : j\in A, s_1\le t \le s_2 \}$. By Lemma \ref{extens}, there exists $f\in {\mathcal D}_V$ such that
\begin{align}
\min_{j\in A} u(j) \ge \epsilon \quad \implies \quad f(u) = \langle 1^A_i,u \rangle \quad \text{and} \quad \lambda(f)(u)=\lambda^A(i).
\end{align}
Now, the desired result follows from applying {\bf (A)} to function $f$.
\end{proof}

We now introduce some elements from linear algebra that we shall use in the proof of Lemma \ref{absort}. Let $S$ be a nonempty proper subset of $V$ and denote $B=V\setminus S$. Consider the set of vectors 
\begin{align}
e^B_j=1^{S\cup\{j\}}_j, \quad j\in B.
\end{align}
Notice that, for any $j,k\in B$, we have
\begin{align}
\langle {\boldsymbol v}_k , e^B_j \rangle = 0, \;\text{if $k\not = j$} \quad \text{and} \quad \langle {\boldsymbol v}_j, e^B_j \rangle < 0.
\end{align}
Therefore, $(e^B_j)_{j\in B}$, is linearly independent and, for each $k\in B$, we may write
\begin{align}\label{desck}
1_{\{k\}} = \sum_{j\in B} \alpha^j_k e^B_j, \quad \text{where} \quad \alpha^j_k=\frac{ \langle {\boldsymbol v}_j , 1_{\{k\}} \rangle }{ \langle {\boldsymbol v}_j , e^B_j \rangle } \cdot
\end{align}
For all $j,k\in B$, it is simple to verify that
\begin{align}\label{desal}
\alpha^k_k\ge 0 \quad \text{and} \quad \alpha^j_k\le 0, \; \text{if $j\not = k$}.
\end{align}

\begin{lemma}\label{absort}
Fix some $\zeta\in C(\R_+, \Sigma)$. Denote $S= \mathcal S(\zeta_0)$ and
\begin{align}
T=\inf\big\{t\ge 0 : \textstyle{ \prod_{i\in S}\zeta_t(i)=0}\big\}.
\end{align}
If $\zeta$ satisfies condition {\bf (A)} and $S$ is $r$-absorbing then
\begin{align}
\zeta_t(j)=0, \quad \forall j\in V\setminus S \quad \text{and} \quad t\in [0,T).
\end{align}
\end{lemma}
\begin{proof}
Assume that $B:=V\setminus S$ is nonempty, otherwise the claim is trivial. Clearly,
\begin{align}
\langle e^B_j , 1 \rangle \ge 1, \quad \forall j\in B.
\end{align}
Define, for each $j\in B$ and $u\in \Sigma$,
\begin{align}
h_j(u) = \frac{\langle e^B_j, u \rangle }{\langle e^B_j , 1 \rangle}, \quad \text{so that} \quad 0\le h_j(u)\le 1.
\end{align}
Also, denote $\Vert u\Vert_h :=\max_{j\in B} h_j(u)$ and observe that 
\begin{align}\label{jq0}
\Vert u\Vert_h=0 \quad \iff \quad u(j)=0, \;\forall j\in B.
\end{align}
Let us assume that the claim is false and proceed by contradiction. By \eqref{jq0},
\begin{align}
\exists \,\delta>0 \quad \text{such that} \quad \sup_{0\le t < T} \Vert \zeta_t \Vert_h > \delta.
\end{align}
Define $s_2 := \inf\{ 0\le t <T : \Vert \zeta_t \Vert_h \ge \delta\}$. By continuity of $\zeta$,
\begin{align}\label{desh}
s_2>0, \quad h_j(\zeta_{s_2}) \le \delta, \;\forall j\in B \quad \text{and} \quad h_{k}(\zeta_{s_2})=\delta, \;\text{for some $k\in B$}.
\end{align}
Let us now prove that $\zeta_{s_2}(k)\ge \delta$. By using \eqref{desck}, we have
\begin{align}
\zeta_{s_2}(k) = \langle 1_{\{k\}}, \zeta_{s_2} \rangle  = \sum_{j\in B} \alpha^j_k \langle e^B_j , \zeta_{s_2} \rangle = \sum_{j\in B} \alpha^j_k \langle e^B_j , 1 \rangle h_j(\zeta_{s_2}).
\end{align}
By \eqref{desh} and \eqref{desal}, the last expression is bounded below by
\begin{align}
\delta \sum_{j\in B} \alpha^j_k \langle e^B_j , 1 \rangle = \delta \langle 1_{\{k\}} , 1 \rangle = \delta.
\end{align}
Now, $\zeta_{s_2}(k)\ge \delta$ guarantees the existence of some $s_1\in (0,s_2)$, so that
\begin{align}
\zeta_t(i) > 0, \quad \text{for all $t\in [s_1,s_2]$ and $i\in S \cup\{k\}$}.
\end{align}
Therefore, by Lemma \ref{linear},
\begin{align}
\langle e^B_k,\zeta_{s_2} \rangle = \langle e^B_k,\zeta_{s_1}\rangle + \lambda^{S\cup \{k\}}(k)(s_2-s_1).
\end{align}
Since $S$ is $r$-absorbing then $\lambda^{S\cup \{k\}}(k)\le 0$ and so 
\begin{align}
\langle e^B_k,\zeta_{s_1}\rangle \ge \langle e^B_k,\zeta_{s_2} \rangle \quad \implies \quad h_k(\zeta_{s_1}) \ge h_k(\zeta_{s_2}) = \delta. 
\end{align}
This contradicts the minimality of $s_2$. We are done.
\end{proof}


\subsection{Starting at $\Sigma_{abs}$}

In this subsection we shall prove that any solution of the $(\lambda,\mathcal D_V)$-problem starting at $u\in\Sigma_{abs}$, must coincide with $\zeta^u$ constructed in \eqref{zetaudef}. Our strategy will be to iterate the result we state in the following lemma.

\begin{lemma}\label{fp}
Suppose that $\zeta\in C(\R_+, \Sigma)$ satisfies condition {\bf (A)}
and that $S:= \mathcal S(\zeta_0)$ is $r$-absorbing. Define,
\begin{align}
\hat\lambda^S(i) :=
\left\{
\begin{array}{cl}
\lambda^S(i), & \text{for $i\in S$}, \\
0, & \text{otherwise},
\end{array}
\right.
\quad \text{and} \quad T:=\inf\big\{t\ge 0 : \zeta_0 + t \hat \lambda^S \not \in \Sigma \big\},
\end{align}
where $\inf \varnothing = \infty$. Then
\begin{align}\label{ww0}
\zeta_t = \zeta_{0} + t\hat\lambda^S, \quad \text{for all $t\in [0,T)$.}
\end{align}
Moreover, if $T<\infty$ then $\mathcal S(\zeta_T)$ is $r$-absorbing and a proper subset of $\mathcal S(\zeta_0)$.
\end{lemma}

\begin{proof}
Consider
\begin{align}
\hat T:=\inf\big\{t\ge 0 : \textstyle{ \prod_{i\in S} \zeta_t(i)=0}\big\}, \quad \text{where} \quad \inf \varnothing = \infty.
\end{align}
By continuity of $\zeta$, we have $\hat T>0$. Fix an arbitrary $T'< \hat T$. Since $\zeta_t(i) >0$, for all $i\in S$ and $t\in [0,T']$ then, by Lemma \ref{linear},
\begin{align}\label{ec0}
\langle 1^S_i,\zeta_t \rangle = \langle 1^S_i,\zeta_{0}\rangle + t\lambda^S(i), \quad \text{for all $i\in S$ and $t\in [0,T']$.}
\end{align}
By Lemma \ref{absort}, we conclude from \eqref{ec0} that
\begin{align}\label{sd1}
\zeta_t = \zeta_{0} + t\hat\lambda^S, \quad \text{for all $t\in [0,\hat T)$.}
\end{align}
We now show that $\hat T=T$. By \eqref{sd1} and definition of $\hat T$ we have, for all $t<\hat T$,
\begin{align}
\textstyle{\prod_{i\in S}}\big(\zeta_0(i)+t\hat\lambda^S(i)\big) > 0 \quad \implies \quad \zeta_{0} + t\hat\lambda^S \in \Sigma.
\end{align}
Therefore, $\hat T\le T$. So, it remains to analyse the case in which $\hat T<\infty$. By definition of $\hat\lambda^S$ and $\hat T$, we have
\begin{align}\label{psub}
\mathcal S(\zeta_{\hat T})\subseteq S \quad \text{and} \quad S \setminus \mathcal S(\zeta_{\hat T}) = \{ i\in S : \zeta_{\hat T}(i)=0 \} \not = \varnothing.
\end{align}
In addition, for any $i\in S \setminus \mathcal S(\zeta_{\hat T})$ we have
\begin{align}\label{ha37}
0= \zeta_{\hat T}(i) = \zeta_{0}(i) + \hat T\lambda^S(i) \quad \implies \quad \lambda^S(i) = - \frac{\zeta_{0}(i)}{\hat T} < 0.
\end{align}
Hence, for all $i\in S \setminus \mathcal S(\zeta_{\hat T})$ and $t>\hat T$ we have $\zeta_0(i)+ t \lambda^S(i)<0$. That is, $T\le \hat T$ and so
\begin{align}\label{ww}
T = \inf\big\{t\ge 0 : \zeta_0 + t \hat \lambda^S \not \in \Sigma \big\} = \inf\big\{t\ge 0 : \textstyle{ \prod_{i\in S} \zeta_t(i)=0}\big\}.
\end{align}
Now, \eqref{ww0} follows from \eqref{sd1}. For the final assertion, suppose that $T<\infty$. \eqref{psub} implies that $\mathcal S(\zeta_{T})$ is a proper subset of $S$. By Lemma \ref{a37} (proved in Subsection \ref{ssa37}) and \eqref{ha37} we conclude that $\mathcal S(\zeta_T)$ is $r$-absorbing.
\end{proof}

\begin{lemma}\label{pabc1}
For each $u\in \Sigma_{abs}$, there exists at most one $\zeta\in C(\R_+, \Sigma)$ satisfying {\bf (A)} and $\zeta_0=u$.
\end{lemma}
\begin{proof}
Set $T_0=0$, $v_0=u$, define
\begin{align}
\lambda^{v_0}(i) = \left\{
\begin{array}{cl}
\lambda^{\mathcal S(v_0)}(i), & \text{for $i\in \mathcal S(v_0)$}, \\
0, & \text{otherwise},
\end{array}
\right.
\end{align}
and denote
\begin{align}
T^{v_0}:=\inf\big\{t\ge 0 : v_0 + t \lambda^{v_0} \not \in \Sigma \big\} \quad \text{where} \quad \inf \varnothing = \infty.
\end{align}
By Lemma \ref{fp},
\begin{align}
\zeta_t = v_0 + t\lambda^{v_0}, \quad \text{for all $t\in [0,T^{v_0})$}.
\end{align}
If $T^{v_0}=\infty$, we are done. Otherwise, we set $T_1:=T^{v_0}$, $v_1 := \zeta_{T_1}$, define
\begin{align}
\lambda^{v_1}(i) = \left\{
\begin{array}{cl}
\lambda^{\mathcal S(v_1)}(i), & \text{for $i\in \mathcal S(v_1)$}, \\
0, & \text{otherwise},
\end{array}
\right.
\end{align}
and denote
\begin{align}
T^{v_1}:=\inf\big\{t\ge 0 : v_1 + t \lambda^{v_1} \not \in \Sigma \big\} \quad \text{where} \quad \inf \varnothing = \infty.
\end{align}
Consider the path $\tilde\zeta\in C(\R_+, \Sigma)$ defined by
\begin{align}
\tilde\zeta_t = \zeta_{T_1+t}, \quad t\ge 0,
\end{align}
It is not difficult to check that $\tilde \zeta$ satisfies {\bf (A)}. 
Lemma \ref{fp} assures that $\mathcal S(\tilde\zeta_0)$ is $r$-absorbing. Therefore, we may apply Lemma \ref{fp} to $\tilde \zeta$ and conclude that
\begin{align}
\zeta_{t} = v_1 + (t-T_1)\lambda^{v_1}, \quad \text{for all $t\in [T_1,T_1 + T^{v_1})$}.
\end{align}
If $T^{v_1}=\infty$, we are done. Otherwise, we set $T_2:=T_1+T^{v_1}$, $v_2 := \zeta_{T_2}$ and repeat the same procedure. Since $\mathcal S(v_0) \supsetneq \mathcal S(v_1)\supsetneq \dots$ this procedure must stop. We end up obtaining a finite sequence of pairs
\begin{align}
(T_n,v_n)\in \R_+\times \Sigma, \quad n=0,1,\dots,f,
\end{align}
and vectors
\begin{align}
\lambda^{v_n}(i) = \left\{
\begin{array}{cl}
\lambda^{\mathcal S(v_n)}(i), & \text{for $i\in \mathcal S(v_n)$}, \\
0, & \text{otherwise},
\end{array}
\right.
\end{align}
so that, for each $n<f$,
\begin{align}
\zeta_t=v_n + (t-T_n)\lambda^{v_n}, \quad \text{for $t \in [T_n,T_{n+1})$},
\end{align}
$\lambda^{v_f}$ is null and $\zeta_t = v_f$ for all $t\ge T_f$.
\end{proof}

Notice that at this point we are not allowed to use $\mathcal A(\cdot)$ because we have not proven Lemma \ref{abinter} yet. Nevertheless, Lemma \ref{fp} and the proof of Lemma \ref{pabc1} make clear that

\begin{remark}\label{rem10}
For $u\in \Sigma_{abs}$, the construction of $\zeta^u$ in \eqref{zetaudef} avoids the use of $\mathcal A(\cdot)$ and is in force. Furthermore, $\zeta^u$ is the only possible $\Sigma$-valued continuous path satisfying {\bf (A)} and $\zeta^u_0=u$.
\end{remark}

Our next step is to relate the $(\lambda,\mathcal D_V)$-problem with the ORP.

\subsection{Proof of Proposition \ref{porp}}

Consider the ORP introduced in Definition \ref{c5} and corresponding to the set of transition probabilities $p$, related to $r$ as in \eqref{probr}. For the vector $\lambda$ given in \eqref{ladef} and each $u\in \Sigma$, recall that $(\zeta^{\lambda,u},\rho^{\lambda,u})$ stands for the unique solution of the ORP with input
\begin{align}
\xi^u_t = u + t\lambda, \quad t\ge 0.
\end{align}

\begin{lemma}\label{zlu}
For all $u\in \Sigma$, the path $\zeta^{\lambda,u}$ satisfies {\bf (A)}.
\end{lemma}
\begin{proof}
Fix an arbitrary $f\in \mathcal D_V$. To keep notation simple, denote $\zeta=\zeta^{\lambda,u}$ and $\rho=\rho^{\lambda,u}$. By $ii)$ in Definition \ref{c5} we have that each $\zeta(j)$ is a function of bounded variation and
\begin{align}
d\zeta_t(j) = \lambda(j)dt + \sum_{i\in V} \bs w_i(j)d\rho_t(i), \quad \text{for $j\in V$.}
\end{align}
We may now apply the chain rule for bounded variation functions (see e.g. Theorem 3.96 in \cite{ambrosio-etal}) to get
\begin{align}\label{AB4}
df(\zeta_t) \; = \; \lambda(f)(\zeta_t)\,dt+  \sum_{i\in V} {\bs w}_i(f)(\zeta_t)\,d \rho_t(i) \; = \; \lambda (f)(\zeta_t) \,dt - \sum_{i\in V} \frac{{\bs v}_i(f)(\zeta_t)}{r(i)}  \,d \rho_t(i),
\end{align}
where, for the last equality, we have used the relation
\begin{align}
\bs v_i= - r(i) \bs w_i, \quad \text{with} \quad r(i)=\sum_{j\in V}r(i,j), \quad \text{for each $i\in V$}.
\end{align}
From definition of $\mathcal D_V$ and \eqref{zdr}, it follows that
\begin{align}\label{cq1}
\frac{{\bs v}_i(f)(\zeta_t)}{r(i)} 1_{\{ \zeta_t(i)=0 \}} \,d \rho_t(i) = 0 \quad \text{and} \quad \frac{{\bs v}_i(f)(\zeta_t)}{r(i)} 1_{\{ \zeta_t(i)>0 \}} \,d \rho_t(i) = 0,
\end{align}
for each $i\in V$. Putting \eqref{AB4} and \eqref{cq1} together we get the desired result.
\end{proof}

As an immediate consequence of this lemma and Remark \ref{rem10}, we conclude that,

\begin{remark}\label{pud}
If $u\in \Sigma_{abs}$ then $\zeta^{\lambda,u}$ is the unique solution of the $(\lambda,\mathcal D_V)$-problem starting at $u$. Furthermore, $\zeta^{\lambda,u}=\zeta^u$, $\forall u\in \Sigma_{abs}$.
\end{remark}

Recall that we have already proved Lemma \ref{a37} and so Lemma \ref{lproma} is in force at this point.

\begin{lemma}\label{lbsa}
Let $S$ be nonempty and $A$ be $r$-absorbing such that 
\begin{align}\label{bsa}
S\subseteq A \quad \text{and} \quad \lambda^{A}(j)>0, \quad \forall j\in A\setminus S.
\end{align}
For any $u\in \Sigma$, such that $\mathcal S(u)=S$, there exists some $\delta>0$ such that
\begin{align}
\mathcal S(\zeta^{\lambda,u}_t)=A, \quad \forall t\in (0,\delta].
\end{align}
\end{lemma}
\begin{proof}
If $S=A$ then, by Remark \ref{pud}, $\zeta^{\lambda,u}=\zeta^{u}$ and the assertion follows from the construction of $\zeta^{u}$. Suppose now that $S\subsetneq A$. Define
\begin{align}
\hat \lambda^{A}(i) = \left\{
\begin{array}{cl}
\lambda^{A}(i), & \text{for $i\in A$}, \\
0, & \text{otherwise},
\end{array}
\right.
\end{align}
fix some $u\in \Sigma$ such that $\mathcal S(u)=S$ and consider
\begin{align}
T=\inf\{ t\ge 0 : u + t\hat\lambda^A \not \in \Sigma \}.
\end{align}
By \eqref{bsa} and definition of $\hat\lambda^A$, we have $0<T<\infty$. Fix now some $\delta>0$ and a sequence $s_n>0$, $n\ge 1$, so that
\begin{align}\label{hh0}
\delta + s_n < T, \; \forall n\ge 1 \quad \text{and} \quad s_n\downarrow 0. 
\end{align}
Define the sequences $(u_n)$ and $(T^n)$ as
\begin{align}
u_n:=u+s_n\hat\lambda^A \quad \text{and} \quad T^n=\inf\{ t\ge 0 : u_n + t\hat\lambda^A \not \in \Sigma \}.
\end{align}
By \eqref{bsa} and the fact that $s_n<T$, we have $\mathcal S(u_n)=A$. Then, by Remark \ref{pud}, for all $n\ge 1$, 
\begin{align}\label{px1}
\zeta^{\lambda,u_n}_t= u_n + t \hat\lambda^A, \quad \forall t\in [0,T^n].
\end{align}
Due to the fact that $s_n<T<\infty$, we have $T^n=T-s_n$ and so, by \eqref{hh0},
\begin{align}\label{px2}
\delta < T ^n, \quad \text{for all $n\ge 1$.} 
\end{align}
Finally, from \eqref{px2}, \eqref{px1} and the continuity of $\Gamma$, as stated in Proposition \ref{cymp}, we have
\begin{align}
\zeta^{\lambda,u}_{t} = \lim_{n\to \infty} \zeta^{\lambda,u_n}_{t} = u + t\hat \lambda^A, \quad \forall t\in (0,\delta].
\end{align}
By \eqref{bsa}, $\mathcal S(\zeta^{\lambda,u}_{t}) = A$, for all $t\in (0,\delta]$.
\end{proof}

We may now provide the following characterization of the $r$-absorbing sets.

\begin{lemma}\label{lez0}
A nonempty $A\subseteq V$ is $r$-absorbing if, and only if, for all $u\in \partial_A\Sigma$,
\begin{align}\label{ez0}
\zeta^{\lambda,u}_t\in \partial_A\Sigma, \quad \forall t\ge 0.
\end{align}
\end{lemma}
\begin{proof}
Let us first assume that $A$ is $r$-absorbing and fix some $u\in \partial_A\Sigma$. Denote $S:=\mathcal S(u)$. If $S=A$ then, by Remark \ref{pud}, $\zeta^{\lambda,u}=\zeta^{u}$ and \eqref{ez0} follows from the construction of $\zeta^{u}$. So, we may assume that $S\subsetneq A$ and denote by $|A\setminus S|$ the cardinality of $A\setminus S$. Define a sequence $(u_n)$ as
\begin{align}\label{seun}
u_n(i)=\left\{
\begin{array}{cl}
(1-2^{-n})u(i) & \text{for $i\in S$}, \\
(2^n|A\setminus S|)^{-1} & \text{for $i\in A\setminus S$}, \\
0 & \text{for $i\in V\setminus A$},
\end{array}
\right.
\end{align}
so that $\mathcal S(u_n)=A$, $\forall n\ge 1$, and $u_n\to u$. As we have already noticed, for all $n\ge 1$, we have
\begin{align}\label{aaq1}
\zeta^{\lambda,u_n}_t= \zeta^{u_n}_t\in \partial_A\Sigma, \quad \forall t\ge 0.
\end{align}
By letting $n\to\infty$ in \eqref{aaq1} and using the continuity of $\Gamma$ as stated in Proposition \ref{cymp}, we get \eqref{ez0}.

Assume now that $A$ is not $r$-absorbing. By Lemma \ref{lproma}, there exists an $r$-absorbing set $B$ such that
\begin{align}
A\subsetneq B \quad \text{and} \quad \lambda^{B}(j)>0, \quad \forall j\in B\setminus A.
\end{align}
Take some $u\in \Sigma$, such that $\mathcal S(u)=A$. By Lemma \ref{lbsa}, there exists some $\delta>0$ such that $\mathcal S(\zeta^{\lambda,u}_{\delta})=B$ contradicting \eqref{ez0}. We are done.
\end{proof}

From this, it is now clear that the intersection of two $r$-absorbing subsets is still $r$-absorbing.

\begin{proof}[Proof of Lemma \ref{abinter}]
Let $A$ and $B$ be $r$-absorbing subsets. Since $\varnothing$ is considered $r$-absorbing, we may assume that $S=A\cap B$ is nonempty. Fix an arbitrary $u\in \partial_S\Sigma = \partial_A\Sigma \cap \partial_B\Sigma$. By Lemma \ref{lez0} applied to subsets $A$ and $B$ we have
\begin{align}
\zeta^{\lambda,u}_t\in \partial_S\Sigma, \quad \forall t\ge 0.
\end{align}
Finally, Lemma \ref{lez0} applied to $S$ assures that $S$ is $r$-absorbing.
\end{proof}

We have finally the notion of $\mathcal A(\cdot)$ in force. Let us prove its characterization as stated in Lemma \ref{min-abs}.

\begin{proof}[Proof of Lemma \ref{min-abs}]
Fix a nonempty $S\subseteq V$. By Lemma \ref{a37}, $\mathcal A(S)$ is an $r$-absorbing set satisfying
\begin{align}
S\subseteq \mathcal A(S) \quad \text{and} \quad \lambda^{\mathcal{A}(S)}(j)>0, \quad \forall j\in \mathcal A(S)\setminus S.
\end{align}
Let $A$ be an $r$-absorbing set satisfying \eqref{bsa}. Fix some $u\in \Sigma$ such that $\mathcal S(u)=S$. According to Lemma \ref{lbsa}, there exists some $\delta>0$ such that 
\begin{align}
A = \mathcal S(\zeta^{\lambda,u}_{\delta}) = \mathcal A(S).
\end{align}  
We are done.
\end{proof}

We have finally collected all the results used in Subsection \ref{FLimit} for the construction of
\begin{align}\label{zzu}
\zeta^u\in C(\R_+,\Sigma), \quad \text{for all $u\in \Sigma$}.
\end{align}
We are now ready to prove that $\zeta^u=\zeta^{\lambda,u}$, for all $u\in \Sigma$.

\begin{proof}[Proof of Proposition \ref{porp}]
Fix an arbitrary $u\in \Sigma$, a sequence $s_n>0$, $n\ge 1$, such that $s_n\downarrow 0$, and denote $u_n:=\zeta^u_{s_n}$, $n\ge 1$. By property \eqref{fpr},
\begin{align}\label{dd1}
\zeta^{u_n}_t=\zeta^u_{s_n + t}, \quad \text{for all $t\ge 0$ and $n\ge 1$}.
\end{align}
As we have already noticed in \eqref{abst}, we have $u_n\in \Sigma_{abs}$ and then, by Remark \ref{pud},
\begin{align}\label{dd2}
\zeta^{\lambda,u_n} = \zeta^{u_n}, \quad \forall n\ge 1.
\end{align}
By using \eqref{dd1}, \eqref{dd2} and the continuity of $\Gamma$ as stated in Proposition \ref{cymp}, we get
\begin{align}
\zeta^u_t = \lim_{n\to \infty} \zeta^{u_n}_t = \lim_{n\to \infty} \zeta^{\lambda,u_n}_t = \zeta^{\lambda,u}_t,
\end{align}
for any $t\ge 0$. We are done.
\end{proof}

\subsection{Well-posedness of the $(\lambda,\mathcal D_V)$-problem}

We may finally prove the main result of this section. The uniqueness part of Proposition \ref{promain} below is crucial in our approach to get the fluid limit for the zero-range process.

\begin{proposition}\label{promain}
For any $u\in \Sigma$, the unique solution of the $(\lambda,\mathcal D_V)$-problem that starts at $u$ is $\zeta^{\lambda,u}=\zeta^u$.
\end{proposition}
\begin{proof}
Suppose that $\zeta\in C(\R_+,\Sigma)$ solves the $(\lambda,\mathcal D_V)$-problem and $\zeta_0=u$. According to Remark \ref{pud}, if $u\in \Sigma_{abs}$ then $\zeta=\zeta^u= \zeta^{\lambda,u}$. Assume then that $u\in \Sigma \setminus \Sigma_{abs}$. According to {\bf (B)}, there exists a sequence $s_n\downarrow 0$ so that
\begin{align}\label{zs1}
u_n:=\zeta_{s_n} \in \Sigma_{abs}, \quad \text{for all $n\ge 1$}.
\end{align}
Consider the sequence of paths
\begin{align}
\zeta^{(n)}_t:=\zeta_{s_n+t}, \quad t\ge 0.
\end{align}
It is simple to verify that each $\zeta^{(n)}$ satisfies {\bf (A)}. Therefore, by \eqref{zs1} and Remark \ref{pud} we have $\zeta^{(n)} = \zeta^{\lambda, u_n}$, $\forall n\ge 1$, and so
\begin{align}
\zeta_t = \lim_{n\to \infty} \zeta^{(n)}_{t} = \lim_{n\to \infty} \zeta^{\lambda,u_n}_{t} = \zeta^{\lambda,u}_t, \quad \forall t\ge 0.
\end{align}
In the last equality, we have used the continuity of $\Gamma$ as stated in Proposition \ref{cymp}. We have thus proved uniqueness for the $(\lambda,\mathcal D_V)$-problem. By Lemma \ref{zlu}, it remains to show that $\zeta^{\lambda,u}$ satisfies {\bf (B)}, when $u\in \Sigma\setminus \Sigma_{abs}$. But this follows immediately from the fact that $\zeta^{\lambda,u}=\zeta^u$, Proposition \ref{porp}, and \eqref{abst}.
\end{proof}

\subsection{Remaining results}

We finish this section by proving Lemma \ref{lorp} and Proposition \ref{absconf}, which complete the description of the fluid limit $\zeta^u$, $u\in \Sigma$. 

\begin{proof}[Proof of Lemma \ref{lorp}]
Fix an irreducible set of transition probabilities $p$ on $V$ satisfying $p(i,i)=0$, $\forall i\in V$, and recall the definition of the set of vectors $\bs w_i$, $i\in V$, given in \eqref{dor}. Let $\mu(i)$, $i\in V$ be the invariant distribution corresponding to $p$ so that
\begin{align}\label{wwii}
\sum_{i \in V} \mu(i) \bs w_i = 0
\end{align}
Consider the subspace orthogonal to the constant $1\in \R^V$:
\begin{align}
{\bs 1}^{\perp} := \{ v \in \R^V : \langle 1,v \rangle = 0 \}
\end{align}
and the closed convex cone
\begin{align}
\mathcal C := \Big\{ \sum_{i\in V} a_i \bs w_i : a_i\ge 0, \forall i\in V \Big\} \subseteq {\bs 1}^{\perp}.
\end{align}
The polar cone of $\mathcal C$ respect to ${\bs 1}^{\perp}$ is
\begin{align}
\mathcal C^*:= \big\{ v\in {\bs 1}^{\perp} : \langle v,u \rangle\le 0, \forall u\in \mathcal C \big\}.
\end{align}
If $v\in \mathcal C^*$ then
\begin{align}
\mu(i) \langle \bs w_i , v \rangle \le 0, \quad \forall i\in V.
\end{align}
From this fact and \eqref{wwii} it follows that $\langle \bs w_i , v \rangle = 0,\, i\in V$, and therefore
\begin{align}
v(i)=\sum_{j\in V} p(i,j) v(j).
\end{align}
Since $p$ is irreducible then $v$ is constant. But $v\in {\bs 1}^{\perp}$ and so $v$ is null. We have proved that $\mathcal C^*=\{0\}$, which in turn implies that $\mathcal C = {\bs 1}^{\perp}$. Therefore, there exists some $a_i$, $i\in V$, such that
\begin{align}
- \lambda =  \sum_{i\in V} a_i {\bs w}_i = \sum_{i\in V} ( a_i + s \mu(i)){\bs w}_i, \quad \forall s\in \R,
\end{align}
where the second identity holds by \eqref{wwii}.
By choosing a sufficiently large $s>0$, we get $r(i):=a_i + s \mu(i)>0$, $\forall i\in V$ as desired.
\end{proof}

We now aim to prove Proposition \ref{absconf}. Let $\mu(i)$, $i\in V$ be the invariant distribution corresponding to the transition rates $r$ and denote 
\begin{align}
{\mathcal M}:=\Big\{i\in V:\, \mu(i)=\max_{k\in V}\mu(k)\Big\}.
\end{align}
Since $r$ is irreducible then, for all nonempty $B\subseteq V$, the trace $r^{B}$ is also irreducible and $\mu(i)$, $i\in V$, is an invariant measure for $r^B$, that is
\begin{align}
\sum_{i\in B} \mu(i)r^B(i,j)=\mu(j)\sum_{i\in B}r^B(j,i), \quad \forall j\in B.
\end{align}
Therefore,
\begin{align}\label{stop1}
\lambda^B(j)=\sum_{i\in B}\big[r^B(i,j) - r^B(j,i)\big]=\sum_{i\in B}\Big(1 - \frac{\mu(i)}{\mu(j)}\Big)r^B(i,j), \quad \forall j\in B.
\end{align}
The following lemma is a straightforward consequence of \eqref{stop1}.
\begin{lemma}\label{stoptech}
For any nonempty $B\subseteq V$ we have
\begin{align}
\text{$\lambda^B$ is null} \quad \iff \quad \text{$\mu$ is constant on $B$}.
\end{align}
\end{lemma}
\begin{proof}
Assume that $\lambda^B$ is null. Therefore, for each $j\in B$,
\begin{align}\label{lie1}
\sum_{i\in B}\Big(1 - \frac{\mu(i)}{\mu(j)}\Big)r^B(i,j) = 0
\end{align}
Denote
\begin{align}
{\mathcal M}_B:=\{j\in B:\,\mu(j)=\max_{i\in B}\mu(i)\}.
\end{align}
Suppose that $\mathcal M_B \subsetneq B$. Fix some $i_0\in B\setminus \mathcal M_B$. Since $r^B$ is irreducible, 
\begin{align}\label{lie2}
r^B(i_0, j_0)>0, \quad \text{for some $j_0\in \mathcal M_B$}.
\end{align}
Since 
\begin{align}
\Big(1 - \frac{\mu(i)}{\mu(j_0)}\Big)r^B(i,j_0) \ge 0, \quad \forall i\in B,
\end{align}
then, from \eqref{lie1} and \eqref{lie2}, it follows that $\mu(i_0)=\mu(j_0)$, which contradicts $i_0\in B\setminus \mathcal M_B$. Therefore $\mu$ is constant on $B$. The converse follows immediately from \eqref{stop1}.
\end{proof}

\begin{proof}[Proof of Proposition \ref{absconf}]
We first prove that $a)$ and $b)$ are equivalent. Assume that $a)$ holds. By Lemma \ref{lez0}, $A$ is $r$-absorbing. Besides, by construction of $\zeta^u$, $\lambda^A$ must be null. That proves $a) \Rightarrow b)$. $b) \Rightarrow a)$ is obvious from the construction of $\zeta^u$. Now, let us prove that $b)$ and $c)$ are equivalent. Assuming $b)$, by Lemma \ref{stoptech}, $\mu$ is constant on $A$. To conclude $c)$, it remains to show that $A\cap \mathcal M$ is nonempty. Suppose that $A\cap \mathcal M = \varnothing$ and fix some $j\in \mathcal M$. Then
\begin{align}\label{muij}
1-\frac{\mu(i)}{\mu(j)}>0, \;\forall i\in A.
\end{align}
Since $r^{A\cup\{j\}}$ is irreducible, there exists some $i\in A$ such that $r^{A\cup\{j\}}(i,j)>0$. From this fact, \eqref{muij} and \eqref{stop1}, we have that
\begin{align}\label{jhh}
\lambda^{A\cup\{j\}}(j) = \sum_{i\in A} \left( 1 - \frac{\mu(i)}{\mu(j)} \right) r^{A\cup\{j\}}(i,j)
\end{align}
turns out to be strictly positive, contradicting that $A$ is $r$-absorbing. Hence $b)\Rightarrow c)$. Finally assume that $c)$ holds. By Lemma \ref{stoptech}, $\lambda^A$ is null. For $j\in V\setminus A$, we may apply \eqref{jhh} and the fact that $\mu(j)\le \mu(i)$, $\forall i\in A$, to conclude that
\begin{align}
\lambda^{A\cup\{j\}}(j) \le 0.
\end{align}
Thus $A$ is $r$-absorbing. By the fact that $\lambda^{A}(i)=0$, we conclude that $\zeta_t^u=u,\ t\geq 0$. We are done.
\end{proof}

\section{Fluid limit for the zero-range process}\label{s2}

We will denote by $L_N$ the generator of the Markov process $\zeta^N$, given in \eqref{zeta}, which acts on a test function $f:\Sigma_N\to \R$ by
\begin{align}\label{L-N}
L_N f(u)=N\sum_{i,j \in V} g(Nu(i))\,r(i,j)\,\big[ f\big(u+\frac{1_{\{j\}}-1_{\{i\}}}{N}\big)-f(u)\big],\quad u \in \Sigma_N.
\end{align}

\subsection{Tightness}\label{tight}

\begin{proof}[Proof of Proposition \ref{tight2}]
Since $\Sigma_N$ is finite, for all $N\ge 1$, then, in order to prove Proposition \ref{tight2}, it is enough to consider an arbitrary sequence $u_N\in \Sigma_N$, $N\ge 1$ and show that 
\begin{align}\label{talt}
\text{$(\P^N_{u_N})_{N\in \N}$ is tight and every limit point is supported on $C(\R_+,\Sigma)$}.
\end{align}
Since $\Sigma$ is compact, by Theorems 13.2 and 13.4 in \cite{billing-conv}, assertion \eqref{talt} follows from,
\begin{equation}\label{s20}
\lim_{\delta\rightarrow 0}\limsup_{N\rightarrow \infty}P\big[\sup_{\substack{s,\,t<T, \\ |s-t|<\delta}}\|\zeta^N_t-\zeta^N_s\|>\epsilon\big]=0, \quad \text{for all $\epsilon>0$ and $T>0$,}
\end{equation}
where $\zeta^N$ is starting at $u_N$. Let us denote $f_j(u)=u(j)$, $j\in V$, $u\in \Sigma$. For each $j\in V$, we have
\begin{align}\label{h3}  
\zeta_t^N(j)=\zeta_0^N(j)+\int_0^t L_Nf_j(\zeta_r)\,dr+M^{N,j}_t,\quad t\ge 0,
\end{align} 
where $M^{N,j}$ is a martingale with respect to the filtration generated by $\zeta^N$. By using \eqref{h3}, we can bound the probability in \eqref{s20} by 
\begin{equation}\label{s21}
\sum_{j\in V}P\Big[\sup_{\substack{s,\,t<T, \\ |s-t|<\delta}}\big|\int_s^t L_Nf_j(\zeta_r)\,dr\big|\geq\frac{\epsilon}{2 |V|}\Big]+\sum_{j\in V} 
P\Big[\sup_{0\leq s\leq T}|M^{N,j}_{s}|\geq\dfrac{\epsilon}{4|V|}\Big],
\end{equation}
where $|V|$ denotes the cardinality of $V$. Now, it is clear that
\begin{equation*}\label{e1}
L_N f_j(u)=\sum_{i\in V}\Big[g(Nu(i))\,r(i,j) - g(Nu(j))r(j,i) \Big]
\end{equation*} 
is bounded and so, the first term in \eqref{s21} vanishes as $\delta\to 0$. On the other hand, by Doob's and Chebyshev's inequalities, the probability in the second term of \eqref{s21} is bounded by
\begin{align}\label{s22}
\frac{2^6\,|V|^2}{\epsilon^2} E\big[(M^{N,j}_{T})^2\big] = \frac{2^6\,|V|^2}{\epsilon^2} E\Big[\int_{0}^{T}\big(L_N (f_j)^2 (\zeta^N_s)-2\zeta^N_s(j) L_N f_j(\zeta^N_s)\big)ds\Big].
\end{align}
The above equality is shown, for instance, in \cite{KL}, Appendix 1.5. An elementary computation shows that 
\begin{align}\label{s222}
L_N (f_j)^2 (u)-2 u(j) L_N f_j(u) \le \frac{C}{N}, \quad \forall u\in \Sigma_N,
\end{align} 
for a positive constant $C$. Therefore, by \eqref{s22} and \eqref{s222} the second term in \eqref{s21} vanishes as $N\to \infty$.
\end{proof}

\subsection{Convergence of $L_N$ on $\mathcal D_V$}

\begin{lemma}\label{aux}
For any $f\in {\mathcal D}_V$ we have
\begin{align*}
\lim_{N\to\infty}\sup_{u\in\Sigma_N}\big|L_N f(u)-\lambda(f)(u)\big| =0.
\end{align*}
\end{lemma}
\begin{proof}
Fix $f\in {\mathcal D}_V$. For each $u\in \Sigma_N$, we approximate $f$ by its first-order Taylor polynomial to get
\begin{align}
L_N f(u) = N\sum_{j,\,k\in V} g(Nu(j))\,r(j,k)\,\big[\big\langle\nabla f(u),\frac{1_{\{k\}}-1_{\{j\}}}{N}\big\rangle + \frac{R^N_{j,k}(u)}{N}\big]
\end{align}
where
\begin{align}\label{Tres}
R^N_{j,k}(u) = \langle \nabla f(v)-\nabla f(u), 1_{\{k\}}-1_{\{j\}}\rangle,
\end{align}
for some $v$, satisfying $\|v-u\|\le\frac2N$. Since $\nabla f$ is continuous on the compact set $\Sigma$,
\begin{align}\label{can0}
\lim_{N\to \infty} \sup_{u\in \Sigma} |R^N_{j,k}(u)| = 0, \quad \text{for each $j,k\in V$.}
\end{align}
Denote 
\begin{align}
R^N(u) := \sum_{j,\,k\in V}g(Nu(j))\,r(j,k)\,R^N_{j,k}(u)
\end{align}
so that
\begin{align}\label{can1}
L_N f(u) - \lambda(f)(u) = \sum_{j\in V} (g(Nu(j))-1)\bs v_j(f)(u) + R^N(u).
\end{align}
Let $M_1$ and $M_2$ be positive constants such that 
\begin{align}\label{def1}
\max_{j\in V}\sup_{u \in \Sigma}|\bs v_j(f)(u)|<M_1 \quad \text{and} \quad \sup_{n\in \N_0} |g(n) -1|<M_2.
\end{align}
For an arbitrary $\epsilon>0$, there exists some $\delta>0$ such that, for all $j\in V$,
\begin{align}\label{def2}
|u(j)|<\delta \quad \implies \quad | \bs v_j(f)(u) | < \frac{\epsilon}{M_2},
\end{align}
and there exists some $N_0\in \N$ such that 
\begin{align}\label{def3}
n\ge N_0 \delta \quad \implies \quad |g(n)-1| < \frac{\epsilon}{M_1}.
\end{align}
From \eqref{def1}, \eqref{def2} and \eqref{def3}, it follows that
\begin{align}\label{can2}
\lim_{N\to \infty} \sup_{u \in \Sigma} |g(Nu(j))-1||\bs v_j(f)(u)| = 0, \quad \forall j\in V.
\end{align}
Finally, by using \eqref{can2} and \eqref{can0} in equation \eqref{can1} we obtain the desired result.
\end{proof}

\subsection{Fluid limits satisfy condition (A)}\label{mgle-prblm}

Let us fix a sequence $u_N$, $N\ge 1$, so that $u_n\to u$. In virtue of Proposition \ref{tight2}, we may assume that 
\begin{align}
\text{$\P^N_{u_N}$ converges weakly to some probability $\P$ in $C(\R_+,\Sigma)$.}
\end{align}
Thanks to the Skorokhod representation theorem, there exists a sequence of random paths $\zeta^N$, $N\ge 1$, and $\zeta$ so that
\begin{align}
\zeta^N \sim \P^N_{u_N}, \quad  \zeta\sim \P
\end{align}
and $\zeta^N$ converges to $\zeta$ in the Skorokhod topology, almost surely. But, since $\zeta$ is almost surely continuous, then (see, for instance, Proposition 1.17, \S~Vl in \cite{jacod-shiryaev})
\begin{align}\label{uocz}
\text{ $\zeta^N$ converges to $\zeta$ uniformly in compact subsets of $\R_+$},
\end{align}
with probability one.

\begin{proposition}\label{conv} 
The random path $\zeta$ satisfies condition {\bf (A)}, almost surely.
\end{proposition}

\begin{proof}
Let $f\in {\mathcal D}_V$. Let us first prove that 
\begin{align}
M^f_t = f(\zeta_t)-f(\zeta_0)-\int_0^t \lambda(f)(\zeta_s)\, ds, \quad t\ge0,
\end{align}
determines a martingale with respect to the filtration generated by $\zeta$:
\begin{align}
\mathcal F_t := \sigma(\zeta_s:0\le s\le t), \quad t\ge 0. 
\end{align}
Fix some $s\ge 0$, $k\in \N$, times $(s_1,s_2,\dots,s_k)\in[0,s]^k$, continuous functions $\Psi_i:\Sigma\to \R$, $i=1,\dots,k$ and define $\Psi(X)=\Psi_1(X_{s_1}) \Psi_2(X_{s_2}) \cdots \Psi_k(X_{s_k})$, for $X\in D(\R_+,\Sigma)$. In virtue of \eqref{uocz}, we have that
\begin{align}\label{tgtg}
E\Big[\Psi(\zeta)\Big( f(\zeta_t)-f(\zeta_s)-\int_s^t \lambda(f)(\zeta_r)\, dr\Big) \Big]
\end{align}
is the limit of
\begin{align}
E\Big[\Psi(\zeta^N)\Big( f(\zeta^N_t)-f(\zeta^N_s)-\int_s^t \lambda(f)(\zeta^N_r)\, dr\Big) \Big], \quad \text{as $N\to\infty$.}
\end{align}
But, for all $N\ge 1$, we have
\begin{align}
E\Big[\Psi(\zeta^N)\Big( f(\zeta^N_t)-f(\zeta^N_s)-\int_s^t L_N f(\zeta^N_r)\, dr\Big) \Big] = 0.
\end{align}
So, by Lemma \ref{aux}, the expected value in \eqref{tgtg} vanishes. Then, we conclude that $M^f_t$, $t\ge 0$ is a $(\mathcal F_t)$-martingale, for any $f\in \mathcal D_V$. Now, $\mathcal D_V$ is an algebra, that is
\begin{align}
f,g\in \mathcal D_V \quad \implies \quad fg\in \mathcal D_V.
\end{align}
Therefore, (see for instance Lemma 6.7 in \cite{BCL}) for any $f,g\in \mathcal D_V$, we have that
\begin{align}
M^f_t M^g_t - \int_0^t \Gamma(f,g)(\zeta_s)\,ds, \quad t\ge 0,
\end{align}
is a $(\mathcal F_t)$-martingale, where
\begin{align}
\Gamma(f,g) = \lambda(fg) - g\lambda(f) - f \lambda(g).
\end{align}
By Leibniz rule, $\Gamma$ vanishes and so $E[(M^f_t)^2]=0$, for all $t\ge 0$ and $f\in \mathcal D_V$. We conclude that
\begin{align}
f(\zeta_t)-f(\zeta_0)-\int_0^t \lambda(f)(\zeta_s)\, ds = 0, \quad \forall t\ge0,
\end{align}
almost surely, as desired.
\end{proof}

\subsection{Coupled zero-range processes}\label{coupling}

We introduce a coupling method that will allow us to construct and compare zero-range processes associated to different rate functions, for all initial particle configurations, in the same probability space.
\vspace{1mm}

{\bf Graphical representation\ \ }   
Independently for each site $i \in V$, consider an intensity $1$-Poisson process $\Gamma^i$ on the first quadrant $\{(t,u)\in [0,\infty)\times[0,\infty)\}$, and an independent sequence of i.i.d.~random variables $\big(Y_n^i\big)_{n\ge 1}$ taking values in $V$ with distribution $(p(i,j))_{j\in V}$, with $p(i,j)=\frac{r(i,j)}{r(i)}$, $(r(i,j))_{i,j \in V}$ the rates of the random walk. Let $\Gamma=(\Gamma^i)_{i\in V}$. We will now give an explicit construction of the zero-range process $(\eta_t,\,t\ge 0)$ associated to an initial configuration 
$\eta_0$, a bounded rate function $\gamma(n)$ and the underlying $\big(r(i,j)\big)_{i,j\in V}$ random walk. Assume that the zero-range process $(\eta_s,\ s<t)$ has been built up to time $t-$, and that up to time $t-$ exactly $m$ particles have exited site $i$. By the fact that $(\gamma(n))_{n\ge 0}$ is bounded, $m<\infty$ with probability $1$. Then, if the Poisson point process $\Gamma^i$ has an atom at $(t,u)$ and $\eta_{t-}(i)\ge 1$,  the site $i$ will lose a particle if $r(i)\gamma\big(\eta_{t-}(i)\big)\ge u$. If that is the case, the particle will jump to the position $Y^i_{m+1}$.

The sequence of i.i.d random variables $\big(Y_n^i\big)_{n\ge 1}$ associated to each site $i\in V$ works as an ordered stack of instructions prescribing the next position of each particle that jumps out of $i$, independently of the time at which the jump is performed. 

	\begin{remark}\label{count}
		{\rm Let $(\eta_t,\,t\ge 0)$ and $(\sigma_t,\,t\ge 0)$ be coupled versions of zero-range processes with rates $(g(n)\big)$ as in \eqref{rate-f}-\eqref{rate-fbis}, and $\big(\one_{n\ge 1}\big)$, respectively, such that 
		their initial configurations $\eta_0,\,\sigma_0 \in \N_0^V$ satisfy 
	\begin{align}\label{ibound}
	\eta_0(i)\ge \sigma_0(i),\, i\in V.
	\end{align}

		Given $i\in V$, let
	\begin{align}\label{arrivals}
	&J_t(\eta,i)=\#\{0\le s\le t,\, \eta_s(i)=\eta_{s-}(i)+1\}\\[1mm]
	&\hspace{4cm} \text{and}\quad  J_t(\sigma,i)=\#\{0\le s\le t,\, \sigma_s(i)=\sigma_{s-}(i)+1\}\notag
	\end{align}
		be the discrete processes counting the number of particles that arrive at $i$ over the interval $[0,t]$, for the $\eta$ and $\sigma$ processes, $t\ge 0$. With the graphical construction above 
		and initial configurations as in \eqref{ibound}, the process $(\eta_t,\,t\ge 0)$ goes through the instructions $(Y_n^i)_{i\in V,\,n\ge 1}$ faster that $(\sigma_t,\,t\ge 0)$, and
	\begin{align}\label{d2}
	J_t(\sigma,i)\le J_t(\eta,i)\quad\text{for all }t\ge 0, \text{ for a.e. realization of } (\eta_t,\,t\ge 0),\,(\sigma_t,\,t\ge 0),\,i\in V.
	\end{align}
	}
	 \end{remark}
 We now consider a family of particularly simple initial configurations, 
so that for a process started from such a configuration, coordinates that are initially non-zero remain positive for all times, and they operate as sources. Given a set $S\subset V$, define $\theta_S \in (\N_0\cup\{+\infty\})^V$ by
	\begin{align}\label{oneA}
	\theta_S(i)=\left\{\begin{array}{ll}
	0\, &i\in V\setminus S,\\
	+\infty\,&i\in S.
	\end{array}
	\right.
	\end{align}\
We also consider extended rate functions: for $\gamma:\N_0\to \R_{\ge 0}$, let 
	\begin{align}\label{d1}
	\bar{\gamma}(n)=\left\{\begin{array}{ll} \gamma(n),\,&n\in \N_0, \\ 1,\,&n=+\infty. \end{array}\right.
	\end{align}

We provide a proof of the following result to keep the presentation self-contained. Alternatively, it is a consequence of the results in \cite{chen-mandelbaum-fluid}.

	\begin{lemma}\label{queue}
	Let $S\subset V$ be non $r$-absorbing and $C:=\mathcal{A}(S)\setminus S$, $\mathcal A(S)$ the set in Definition \ref{def-min-abs}.
		Let $(\sigma_t^{(S)},\,t\ge 0)$ be the zero-range process with initial configuration $\theta_S$ and extended rates $\bar{\one}_{n\ge 1}$, constructed with the graphical 			representation. Let $\delta>0$. There exists $0<t_\delta<\infty$ such that 
			\begin{align}\label{d3}
			P\Big(\,\inf_{t_\delta\le t}\,\frac{\sigma^{(S)}_t(j)}{t}\ge \frac{\lambda^{\mathcal A(S)}(j)}{2}\,\Big)\ge 1-\delta,\qquad j\in C.
			\end{align}
	\end{lemma}
\begin{proof}
By the martingale decomposition of the zero-range process, 
	\begin{align}\label{ito}
	\sigma^{(S)}_t(j)=\lambda(j) t+M_t(j)+[(I-p^t)R_t](j),\qquad t\ge 0,\, j\in V\setminus S,
	\end{align}  
where $\big(M_t(j),t\ge 0\big)$ is a martingale, $p^t$ is the transpose of the matrix $p$ with entries $p_{ij}=p(i,j)$, $i,j\in V$, and $R_t=\big(R_t(k)\big)_{k\in V}$ is defined as 
	\[
	R_t(k)=r(k) \int_0^t \one_{\sigma^{(S)}_s(k)=0}\,ds,\quad k\in V.
	\] 
Let now $n\in \N$. Changing variables in \eqref{ito} we get 
	\begin{align}\label{ito2}
	\frac{\sigma^{(S)}_{2^nt}(j)}{2^n}=\lambda(j) t+\frac{M_{2^nt}(j)}{2^n}+[(I-p^t)R^n_t](j),\qquad t\ge 0,\,j\in V\setminus S,
	\end{align}
$R^n_t(k):=r(k)\int_0^t \one_{\sigma^{(S)}_{2^n s}(k)=0}\,ds$. From Chebychev's and Doob's $L^2$ inequalities it follows that for any time $T>0$,
	\begin{align}\label{doob}
	P\Big(\sup_{0\le t\le T}\frac{|M_{2^nt}(j)|}{2^n}\ge \frac{1}{2^{n/3}}\Big)\le K \frac{T}{2^{n/3}},\qquad j\in V\setminus S,
	\end{align}
where $K$ is a positive constant that does not depend on $n$ or $T$.

We will now compare the process $\big(\frac{\sigma^{(S)}_{2^nt}}{2^n},\,t\ge 0\big)$ in \eqref{ito2} with the first coordinate of the solution $\big((\zeta^{\lambda,\theta_S}_t,\rho^{\lambda,\theta_S}_t),\,t\ge 0\big)$ to the ORP  in the orthant $\R_+^{V\setminus S}$ 
	\begin{align}\label{solu}
	\zeta^{\lambda,\theta_S}_t(j)=\lambda(j) t+[\big((I-p^t)_{ij}\big)_{i,j\in V\setminus S}\,\rho^{\lambda,\theta_S}_t](j),\quad t\ge 0, \,j\in V\setminus S.
	\end{align}
It follows from the fact that $p$ is irreducible that $\big(p^t_{ij}\big)_{i,j\in V\setminus S}$ is a non-negative  matrix with zeros on the diagonal and spectral radius strictly less than $1$. By Theorem 1 in \cite{harrison-reiman}, the solution to \eqref{solu} is unique.

On the other hand, by Proposition \ref{porp}, the unique solution $\big((z_t,y_t),\,t\ge 0\big)$ to the ORP in $\Sigma$, 
	\begin{align}\label{4.40}
	z_t=x_0+\lambda t+(I-p^t)y_t,\quad x_0\in \Sigma,\, \mathcal{S}(x_0)=S,
	\end{align}
satisfies 
	\begin{align*}
	&z_t(i)=x_0(i)+\lambda^{\mathcal{A}(S)}(i)\,t,\quad i\in \mathcal{A}(S),\\
	&z_t(j)=0,\quad j\in V\setminus \mathcal{A}(S),
	\end{align*}
for $t\le \tau=\inf\{s\ge 0,\, \prod_{i\in S}(x_0(i)+\lambda^{\mathcal{A}(S)}(i)s) =0\}$. In particular, for $t\le \tau$,
	\begin{align*}
	&(I-p^t)y_t=(\lambda^{\mathcal{A}(S)}-\lambda)t \quad \text{and}\quad y_t(i)=0,\,\quad i\in \mathcal{A}(S),
	\end{align*}
and denoting by $Q^S \in \R^{V\setminus S\times V\setminus S}$ the inverse of the submatrix $\big((I-p^t)_{ij}\big)_{i,j \in V\setminus S}$, it follows that 
	\begin{align}\label{4.41}
	y_t(j)=\Big[ Q^S\,\big(\lambda^{\mathcal{A}(S)}(i)-\lambda(i)\big)_{i\in V\setminus S}\Big](j)\times t, \qquad t\le \tau,\quad j\in V\setminus S.
	\end{align}
Combining representation \eqref{4.41} of $\big(y_t,\, 0\le t\le \tau\big)$ with the properties it satisfies as part of the solution to the ORP in \eqref{4.40}, we conclude that  the vector $Q^S\,\big(\lambda^{\mathcal{A}(S)}(i)-\lambda(i)\big)_{i\in V\setminus S} \in \R^{V\setminus S}$ is such that its coordinates in $\mathcal{A}(S)\setminus S$ vanish, while those in $V\setminus \mathcal{A}(S)$ are non-negative. For $t\ge 0$, let us define
	\begin{align*}
	&\rho_t=Q^S\,\big(\lambda^{\mathcal{A}(S)}(i)-\lambda(i)\big)_{i\in V\setminus S}\times t \qquad \text{ and }\qquad 
	\zeta_t=\lambda t+(\big(I-p^t)_{ij}\big)_{i,j\in V\setminus S}\,\rho_t\,,\\
	\intertext{so that}
	&\zeta_t(i)=\lambda^{\mathcal{A}(S)}(i) \,t, \quad i\in C,\\
	&\zeta_t(j)=0,\quad \qquad j\in V\setminus \mathcal{A}(S).
	\end{align*}
The pair $\big((\zeta_t,\rho_t),\,t\ge 0)$ above verifies the conditions of the ORP \eqref{solu}, and by uniqueness of the solution, we conclude that 
	\begin{align*}
	\zeta^{\lambda,\theta_S}_t(i)=\left\{
	\begin{array}{ll}
	\lambda^{\mathcal A(S)}(i)\, t\,\qquad&i\in C\\[1mm]
	0&i\in V\setminus  \mathcal A(S)
	\end{array}\right.,
	\quad t\ge 0.
	\end{align*}
	
Let $0<t_0<\frac{T}{2}$, and $0<\epsilon<\frac{t_0}{2} \min_{j\in C}\lambda^{\mathcal A(S)}(j)$. 
By the continuity in the input function of the reflection mapping, the processes in \eqref{ito2} and \eqref{solu} are close in the topology of uniform convergence over compact intervals, if the martingale terms are small. That is, there exists $n_0$ so that if $n\ge n_0$,  $\frac{1}{2^{n/3}}$ is small enough that
	\begin{align}\label{d4}
	\sup_{0\le t\le T} \Big|\,\frac{\sigma^{(S)}_{2^n t}(j)}{2^n}-\zeta^{\lambda,\theta_S}_t(j)\,\Big|\le \epsilon,\qquad j\in V\setminus S,
	\end{align}
in the set ${\mathcal O}_n=\medcap_{j\in V\setminus S}\big\{\sup_{0\le t\le T}\frac{|M_{2^nt}(j)|}{2^n}\le \frac{1}{2^{n/3}}\big\}$.	
It follows from \eqref{d4} that for realizations in ${\mathcal O}_n$,
	\begin{align}\label{d5}
	\sup_{t_0\le t\le T} \Big|\,\frac{\sigma^{(S)}_{2^n t}(j)}{2^n t}-\lambda^{\mathcal A(S)}(j)\,\Big|\le \frac{\epsilon}{t_0}\le \frac12 \min_{j\in C}\lambda^{\mathcal A(S)}(j),\quad \text{and}\,
	\inf_{t_0\le t\le T} \frac{\sigma^{(S)}_{2^n t}(j)}{2^n t}\ge \frac{\lambda^{\mathcal A(S)}(j)}{2},
	\end{align}
$j\in C$. From \eqref{d5} and the choice $t_0<\frac{T}{2}$ that for realizations in $\cap_{n\ge n_0} {\mathcal O}_n$, we conclude that
	\begin{align*}
	\inf_{s\ge 2^{n_0}t_0} \frac{\sigma^{(S)}_{s}(j)}{s}\ge \frac{\lambda^{\mathcal A(S)}(j)}{2},\qquad j\in C,
	\end{align*}
with $P\big( (\cap_{n\ge n_0} {\mathcal O}_n)^c \big)\le KT\sum_{n\ge n_0}\frac{1}{2^{n/3}}$. Choose $n_0$ so that this series is bounded by $\delta$, and let $t_{\delta}:=2^{n_0}t_0$ to obtain \eqref{d3}.
\end{proof}

The following result implies that limits of the rescaled zero-range process instantaneously exit non $r$-absorbing boundaries. Let $(\zeta^N_t,\,t\ge 0)$ be the process with law $\P_{\zeta_0}^N$ defined in \eqref{zeta}. 
	\begin{lemma}\label{now-eta}
	Let $S\subset V$ be non-absorbing, $C=\mathcal{A}(S)\setminus S$. For any pair of times $0<s<t$, $a>0$ and $0<b<\frac{\min_{j\in C}\lambda^{\mathcal A(S)}(j)}{12} \frac{t-s}{4}$, 
	\begin{align}\label{4.50}
	\limsup_{N\to \infty}\sup_{\zeta_0^N\in \Sigma_N} \P_{\zeta_0^N}^N\Big(\big\{\inf_{s<l<t}\zeta_l(i)>a, i\in S\big\}\medcap \big\{\inf_{j\in C}\inf_{\frac{s+3t}{4}<l<t}\zeta_l(j) <b \big\}\,\Big)=0.
	\end{align}	
	\end{lemma}
\begin{proof} 
Given an initial configuration $\zeta_0^N \in \Sigma_N$, construct $\zeta^N_l=\frac{\eta^N_{Nl}}{N},\,l\ge 0$, where $(\eta^N_u,\,u\ge 0)$ is obtained with the graphical representation. Fix $\delta>0$ and $j\in C$.

For $u>0$, let $H_u(\Gamma)$ denote the shift of the family of Poisson point processes $(\Gamma^i)_{i\in V}$ by $u$, so that $H_u(\Gamma)^i$ has an atom 
at $(v,x)$ if and only if $\Gamma^i$ has an atom at $(u+v, x)$. Let $\big(\sigma^{(S),u}_{v},\,v\ge 0\big)$ be 
the zero-range process with rates $\bar{1}_{k\ge 1}$ and initial configuration $\theta_S$ obtained from the graphical construction that uses the atoms of $H_u(\Gamma)$.

Given $0<s<t$, on the event $G_a^{S,N}(s,t):=\{\inf_{s<l<t}\zeta^N_l(i)>a, i\in S\}$, all sites $i\in S$ are occupied by at least one $\eta^N$-particle over the period $(Ns,Nt)$, and therefore 
every $\Gamma^i$-atom in the strip $(Ns,Nt)\times[0,1]$ causes the exit of an $\eta^N$-particle from site $i$. That is, no jump attempt used by  $(\sigma^{(S),Ns}_u)$ from sites in $S$ is missed by $(\eta^N_{Ns+u})$, $u\in (0,N(t-s))$, and a bound similar to \eqref{d2} holds for the arrival processes defined 
in \eqref{arrivals}: $J_u(\sigma^{(S),Ns},j)\le J_{Ns+u}(\eta^N,j)-J_{Ns}(\eta^N,j),\,0<u<N(t-s)$.

Define $\epsilon:=\frac{\lambda^{\mathcal{A}(S)}(j)}{2}>0$.
Let $m_0 \in \N$ be such that $g(n)<1+\frac{\epsilon}{2}$ if $n\ge m_0$, and $n_0=n_0(\delta) \in \N$ as in the proof of Lemma \ref{queue}. 
Let ${\mathcal O}(n_0):=\cap_{n\ge n_0}{\mathcal O}_n$, with $\mathcal O_n$ defined below equation \eqref{d4}.  For configurations in $G_a^{S,N}(s,t)$ such that, additionally, $H_{Ns}(\Gamma)\in {\mathcal O}(n_0)$, and times 
$2^{n_0}t_0\le t_1\le u\le N(t-s)$, with $t_0$ as in the proof of Lemma \ref{queue} and $t_1$ that we fix below, we have
	\begin{align}
\hspace{-0.9cm}	\eta^N_{Ns+u}(j)&\ge \eta^N_{Ns}(j)+J_{Ns+u}(\eta^N,j)-J_{Ns}(\eta^N,j)-\#\big\{0\le v\le u,\,\eta^N_{Ns+v}(j)=\eta^N_{Ns+v-}(j)-1\big\}\notag\\[1mm]
	&\ge J_u(\sigma^{(S),Ns},j)-\#\big\{0\le v\le u,\,\eta^N_{Ns+v}(j)=\eta^N_{Ns+v-}(j)-1\big\}\hspace{2.5cm} \text{by \eqref{d2}} \notag\\[1mm]
	&\ge J_u(\sigma^{(S),Ns},j)- \#\big\{0\le v\le u,\,\sigma^{(S),Ns}_v(x)=\sigma^{(S),Ns}_{v-}(j)-1\big\}\label{pri}\\[1mm]
	&\hspace{3.1cm}-H_{Ns}(\Gamma)^j\big([0,2^{n_0}t_0]\times[0,1]\big)\label{seg}\\[1mm]
	&\hspace{3.1cm}-H_{Ns}(\Gamma)^j\big([0,t_1]\times\textstyle{\big[1+\frac{\epsilon}{2},\displaystyle{\max_m g(m)}\big]}\big)\label{terc}\\[1mm]
	&\hspace{3.1cm}-H_{Ns}(\Gamma)^j\big([t_1,u]\times \textstyle{\big[1+\frac{\epsilon}{2},\displaystyle{\max_m g(m)}\big]}\big)\,\one_{\big\{\displaystyle{\min_{t_1\le v\le u} \eta^N_{Ns+v}(j)< m_0}\big\}}\label{cuart}\\
	&\hspace{3.1cm}-H_{Ns}(\Gamma)^j\big([0,u]\times\textstyle{\big[1,1+\frac{\epsilon}{2}\big]}\big).\label{quint}
	\end{align}
The first line on the right above, \eqref{pri}, equals $\sigma^{(S),Ns}_u(j)$. The second line, \eqref{seg}, accounts for Poisson atoms that may have determined jumps for 
$\eta^N_{Ns+\cdot}(j)$ but were missed by $\sigma^{(S),Ns}_\cdot(j)$, because at the time of the jump the queue at $j$ was empty: after time $2^{n_0}t_0$, on the other hand, on the set ${\mathcal O}(n_0)$,
the queue is busy at all times, and no jumps are missed. The expression in \eqref{cuart} counts the atoms of $H_{Ns}(\Gamma)^j$ that determine an exit for $\eta^N_{Ns+\cdot}(j)$ only when there are less than  than $m_0$ particles at the site and $g(\eta^N_{Ns+\cdot}(j))\ge1+\frac{\epsilon}{2}$, and \eqref{terc} and \eqref{quint} count the remaining atoms of 
$H_{Ns}(\Gamma)^j$ that may cause an exit from $\eta^N_{Ns+\cdot}(j)$.

Let now $t_1=t_1(\delta)>\frac{6m_0}{\epsilon}$ be such that 
 	\begin{enumerate}
 	\item[1.] $P\big(H_{Ns}(\Gamma)^j\big([0,2^{n_0}t_0]\times[0,1]\big)>\frac{\epsilon\,t_1}{6}\big)=P\big(\Gamma^j\big([0,2^{n_0}t_0]\times[0,1]\big)>\frac{\epsilon\,t_1}{6}\big)<\delta^2$, where the identity between these probabilities holds due to the translation invariance of Poisson processes, 
	\vspace{1mm}
	\item[2.] $P\big(H_{Ns}(\Gamma)^j\big([0,u]\times\big[1,1+\frac{\epsilon}{2}\big]\big)\le \frac{2}{3} \epsilon u,\, u\ge t_1\big)=P\big(\Gamma^j\big([0,u]\times\big[1,1+\frac{\epsilon}{2}\big]\big)\le \frac{2}{3} \epsilon u,\, u\ge t_1\big)\ge 1-\delta^2$.  Note that by the LLN for Poisson processes $\frac{1}{u}\Gamma^j ([0,u]\times [1,1+\frac{\epsilon}{2}])\xrightarrow[\,u\to \infty\,]{a.s.} \frac{\epsilon}{2}<\frac23 \epsilon$, hence the inequality will hold if $t_1$ is chosen large large enough.
 	\end{enumerate}
Denote
	\begin{align}\label{d6}
	{\mathcal E}(\delta):=\big\{\Gamma \in {\mathcal O}(n_0)\} &\medcap\big\{\Gamma^j\big([0,2^{n_0}t_0]\times[0,1]\big)\le \textstyle{\frac16} \epsilon t_1 \big\} \notag\\[1mm]
	&\hspace{0.5cm} 
	\medcap \big\{\Gamma^j\big([0,u]\times\textstyle{\big[1,1+\frac{\epsilon}{2}\big]}\big)\le \frac{2}{3} \epsilon u,\,  u\ge t_1\big\}\\
	&\hspace{0.5cm} \medcap \big\{\Gamma^j\big( [0,t_1]\times\textstyle{\big[1+\frac{\epsilon}{2}, \displaystyle{\max_m g(m)}\big]}\big)=0\big\}.\notag
	\end{align}
For a configuration $\omega \in G_a^{S,N}(s,t)\cap \{H_{Ns}(\Gamma)\in {\mathcal E}(\delta)\}$, it follows from \eqref{pri}-\eqref{quint}, \eqref{d6} and the choice of $\epsilon$, that
	\begin{eqnarray*}
	\eta^N_{Ns+u}(j)&\ge \sigma^{(S),Ns}_u(j)-\frac16 \epsilon t_1-\frac{2}{3} \epsilon u-H_{Ns}(\Gamma)^j\big([t_1,u]\times \textstyle{\big[1+\frac{\epsilon}{2},\displaystyle{\max_m g(m)}\big]}\big)\,
	\one_{\big\{\displaystyle{\min_{t_1\le v\le u }} \eta^N_{Ns+v}(j)< m_0\big\}}\\
	&\ge \frac16  \epsilon\, u -H_{Ns}(\Gamma)^j\big([t_1,u]\times \textstyle{ \big[1+\frac{\epsilon}{2},\displaystyle{\max_m g(m)}\big]}\big)\,
	\one_{\big\{\displaystyle{\min_{t_1\le v\le u }} \eta^N_{Ns+v}(j)< m_0\big\}},
	\end{eqnarray*}
$t_1\le u\le N(t-s)$. Since $t_1> \frac{6m_0}{\epsilon}$, the first term on the right above is strictly greater than $m_0$ for $u\ge t_1$, and then
	\begin{align}
	\label{final1}
	\eta^N_{Ns+u}(j)\ge \frac16\,\epsilon u\qquad t_1\le u \le N(t-s),\quad \omega\in G_a^{S,N}(s,t)\cap \{H_{Ns}(\Gamma)\in {\mathcal E}(\delta)\} .
	\end{align}
So far, we showed that $\eta^N_{Ns+\cdot}(j)$ grows at least linearly, for a macroscopic time interval, when the configuration belongs to the set $G_a^{S,N}(s,t)\cap \{H_{Ns}(\Gamma)\in {\mathcal E}(\delta)\}$. Among the conditions in definition \eqref{d6}, the one on the last line is the most restrictive. Therefore, in order to extend this result to a set with almost full probability $\P_{\zeta^N_0}^N(G_a^{S,N}(s,t))$, instead of requiring that $\Gamma^j$ have no marks in $[1+\epsilon/2, \max g(m)]$ over the interval $[Ns, Ns+t_1]$, we wait for the first stretch of time having
length $t_1$ where this constraint occurs, and couple with the queue from this time forward. We do this rigorously below.

Consider the stopping time
	\begin{align*}
	\tau^N=Ns+\inf\big\{u\ge t_1,\, H_{Ns}(\Gamma)^j\big([u-t_1,u]\times \textstyle{ \big[1+\frac{\epsilon}{2}, \displaystyle{\max_m g(m)}\big]}\big)=0\big\},
	\end{align*}
so that the shifted process $H_{\tau^N-t_1}(\Gamma)$ satisfies the condition on the third line of the definition of ${\mathcal E}(\delta)$ in \eqref{d6}. Note that $P(\tau^N<\infty)=1$ by Borel-Cantelli's lemma. Let 
	\begin{align*}
	{\Theta}^N(\delta)=\big\{\Gamma=(\Gamma^i)_{i\in V}:\, H_{\tau^N-t_1}(\Gamma)\in {\mathcal E}(\delta) \big\}.
	\end{align*}
By an argument similar to the one leading to \eqref{final1}, for realizations in $G_a^{S,N}(s,t)\cap{\Theta}^N(\delta)$ we get
	\begin{align*}
	\eta^N_{\tau^N+v}(j)\ge \frac16 \,\epsilon\,v\ge \frac{\lambda^{\mathcal A(S)}(j)}{12}\,v, \qquad 0\le v\le \max(Nt-\tau^N,0).
	\end{align*}
If, morever, $\{\tau^N\le N \frac{s+t}{2}\}$, and we write $u=\tau^N+v$, it follows that
	\begin{align*}
	\eta^N_u(j)&\ge \frac{\lambda^{\mathcal A(S)}(j)}{12} \big[u-\tau^N\big]\ge \frac{\lambda^{\mathcal A(S)}(j)}{12} \big(u-N\textstyle{\frac{s+t}{2}}\big),\quad &&N\textstyle{\frac{s+t}{2}}<u<Nt, \notag\\
	\intertext{or, in terms of $\zeta^N_{l}=\frac{\eta^N_{Nl}}{N}$, with $l=\frac{u}{N}$,}
	\zeta^N_l(j)&\ge \frac{\lambda^{\mathcal A(S)}(j)}{12} (l-\textstyle{\frac{s+t}{2}}), &&\textstyle{\frac{s+t}{2}}<l<t.
	\end{align*}
To summarize, $\big\{\inf_{\frac{s+t}{2}<l<t}\,\frac{\zeta^N_l(j)}{l-\frac{s+t}{2}}\ge \frac{\lambda^{\mathcal A(S)}(j)}{12}\big\}$ holds in the event $G_a^{S,N}(s,t)\cap{\Theta}^N(\delta)\cap \{\tau^N\le N \frac{s+t}{2}\}$. Let us estimate the probability of this event. 

Denote  by ${\mathcal F}_{\tau^N}$ the $\sigma$-algebra associated to $\tau^N$. We have
	\begin{align}\label{almost}
	P\big({\Theta}^N(\delta)\big)
	&=E\Big[P\Big(\big\{ (\sigma_u^{S,\tau^N-t_1},\,u\ge 0)\in{\mathcal O}(n_0)\big\}\medcap\big\{H_{\tau^N-t_1}(\Gamma)^j\big([0,2^{n_0}t_0]\times[0,1]\big)
	\le\textstyle{\frac16} \epsilon t_1\big\}
	\notag\\
	&\hspace{4.2cm} \medcap \big\{H_{\tau^N-t_1}(\Gamma)^j\big([0,u]\times\textstyle{\big[1,1+\frac{\epsilon}{2}\big]}\big)\le \frac{2}{3} \epsilon u,\, u\ge t_1\big\}
	\Big|{\mathcal F}_{\tau^N}\Big)\Big] \notag\\
	&=P\Big(\big\{(\sigma^{(S)}_u,\,u\ge 0)\in {\mathcal O}(n_0)\big\}\medcap\big\{\Gamma^j\big([0,2^{n_0}t_0]\times[0,1]\big)\le\textstyle{\frac16}\epsilon t_1\big\}
	\notag\\
	&\hspace{4.2cm} \medcap \big\{\Gamma^j\big([0,u]\times\textstyle{\big[1,1+\frac{\epsilon}{2}\big]}\big)\le \frac{2}{3} \epsilon u,\, u\ge t_1\big\}\Big),
	\end{align}
as the joint distribution of   
	\begin{align}\label{tau-shift}
\hspace{-5mm}	(\sigma^{(S),\tau^N\hspace{-1mm}-t_1}_u,\,u\ge 0),H_{\tau^N\hspace{-1mm}-t_1}(\Gamma)^j\big([0,2^{n_0}t_0]\times[0,1]\big),\  
	\big(H_{\tau^N\hspace{-1mm}-t_1}(\Gamma)^j\big([0,u]\times\textstyle{\big[1,1+\frac{\epsilon}{2}\big]}\big),\,u\ge t_1\big)
	\end{align}
given ${\mathcal F}_{\tau^N}$, is a.s. equal to the joint distribution of
	\begin{align*}
	(\sigma^{(S)}_u,u\ge 0), \ \ \Gamma^j\big([0,2^{n_0}t_0]\times[0,1]\big),\ 
	\big(\Gamma^j\big([0,u]\times\textstyle{\big[1,1+\frac{\epsilon}{2}\big]}\big),u\ge t_1\big).
	\end{align*}
This holds because $\tau^N$ is determined by the distribution of $\Gamma^j$-Poisson atoms in $\R_+\times[1+\frac{\epsilon}{2}, \max g(m)]$, while the processes and the variable in \eqref{tau-shift} depend on the $\Gamma^k$-atoms in $\R_+\times[0,1+\frac{\epsilon}{2}]$, $k\in V$, and the intersection of these regions has Lebesgue measure zero. The claim follows from the independence of the distribution of Poisson atoms in domains having measure zero intersection, and the translation invariance of Poisson processes.

 The events described in the last two lines of \eqref{almost} are determined by the distribution of Poisson marks in sets that have measure zero intersection, hence they are independent. It follows from \eqref{almost}, properties 1. and 2. of $t_1$, and Lemma \ref{queue}, that 
	\begin{align*}
	P\big( {\Theta}^N(\delta)\big)\ge [P({\mathcal O}(n_0))-\delta^2]\times(1-\delta^2)\ge 1-3\delta.
	\end{align*}
On the other hand,
	\begin{eqnarray*}
	P\big(\tau^N\le N\textstyle{\frac{s+t}{2}}\big)&=P\Big(\textstyle{\inf\big\{u\ge t_1,\,\Gamma^j\big([u-t_1,u]\times \big[1+\frac{\epsilon}{2}, \displaystyle{\max_m g(m)}\big]\big)=0\big\}\le \frac{N(t-s)}{2}}\Big)\to 1
	\end{eqnarray*}
hence $P\big(\tau^N\le N\frac{s+t}{2}\big)\ge 1-\delta$ if $N \ge N_\delta$ large enough. For $N\ge N_\delta$ we get
	\begin{align*}
	P\big( G_a^{S,N}(s,t)\cap{\Theta}^N(\delta)\cap \{\tau^N\le N\textstyle{\frac{s+t}{2}}\}\big)\ge P\big( G_a^{S,N}(s,t)\big)-4\delta,
	\end{align*}
and therefore
	\begin{align*}
	\P^N_{\zeta_0^N}\Big(\big\{\inf_{s<l<t}\zeta_l>a,\,i\in S\big\} \medcap \big\{\inf_{\frac{s+t}{2}<l<t} \frac{\zeta_l(j)}{l-\frac{s+t}{2}}<\frac{\lambda^{\mathcal A(S)}(j)}{12} \big\}\Big)\le 4\delta.
	\end{align*}
We restrict the set of times in the second event to $\frac{s+3t}{4}\le l\le t$, recall that $0<b<\frac{\min_{j\in C} \lambda^{\mathcal A(S)}(j)}{12} \frac{t-s}{4}$, and add the probabilities above over $j\in C$, to obtain
	\begin{align}\label{4.51}
	\P^N_{\zeta_0^N}\Big(\big\{\inf_{s<l<t}\zeta_l>a,\,i\in S\big\} \medcap \big\{\inf_{j\in C}\inf_{\frac{s+3t}{4}<l<t} \zeta_l(j)<b\big\}\Big)\le 4 \delta |V|, 
	\end{align}
$|V|$ the cardinality of V. The bound in \eqref{4.51} does not depend on the initial configuration, hence we might take the supremum over $\zeta_0^N$ and then the limit superior as $N\to \infty$. As $\delta$ is arbitrary, the result follows.
\end{proof}

\subsection{Fluid limits satisfy condition (B)}\label{dem-c2}

In this section we prove our main result.

{\it Proof of Theorem \ref{c2}.}
We may assume that the full sequence $\P^N_{u_N}$ converges weakly to a probability $\P$ on $C(\R_+,\Sigma)$.

Let $S_0={\mathcal S}(u)$. If $S_0$ is $r$-absorbing the result follows from Proposition \ref{conv} and Remark \ref{rem10}.

Otherwise, for $\delta>0$, define $S:[0,\delta]\to {\mathcal P}(V)$ by $S_t:={\mathcal S}(\zeta_t)$,  ${\mathcal P}(V)$ the power set of $V$. Let $|A|$ denote the cardinality of $A\in \mathcal P(V)$. Notice that if $(\zeta_t,\,t\ge 0)$ is continuous and $t_0 \in [0,\delta]$ is a local maximum of $|S_t|$, then $S_t$ is constant in an open neighbourhood of $t_0$. Consider the family of intervals
	\begin{align}\label{2.91}
	\mathcal I=\big\{(s,t):\,0\le s<t \le \delta,\, |S_r|=\max_{l\in [0,\delta]}|S_l| \,\,\text{ if } s<r<t\big\},
	\end{align}
and let 
	\begin{align*}
	s_\delta=\min\big\{0\le s: \text{ there is }s<t\le \delta,\, (s,t)\in \mathcal I, \text{ and }t-s=\max_{(s',t')\in \mathcal I}t'-s'\big\}.
	\end{align*}
Since $[0,\delta]$ is a bounded interval, the maximum in the definition above is achieved, and there are finitely many intervals in $\mathcal I$ with this length, hence $s_\delta$ is well-defined. Let $t_\delta$ be the 
right endpoint of the interval associated to $s_\delta$, $(s_\delta, t_\delta) \in \mathcal I$, set $\tau_\delta:=\frac{s_\delta+t_\delta}{2}\in (s_\delta, t_\delta)$. We claim that 
	\begin{align}\label{2.910}
	S_{\tau_\delta}=\mathcal S(\zeta_{\tau_\delta})\in \Sigma_{abs}, \qquad \P\text{-a.s.}.
	\end{align}
	
Consider ${\mathcal N}_\delta=\{\zeta_t,\,t\ge 0, S_{\tau_{\delta}} \text{ is not $r$-absorbing}\}$. By the observation preceding \eqref{2.91}, we have
	\begin{align*}
	{\mathcal N}_\delta \cap C(\R_+, \Sigma) \,\,\subseteq \displaystyle{\bigcup_{\substack{0\le s<t<\delta \\ s,\,t\in \Q\\  G\subseteq V \text{ not $r$-absorbing}}}} \big\{ {\mathcal S}(\zeta_r)=G,\,s\le r\le t\big\},
	\end{align*}
and, recalling that $\P$ is supported on $C(\R_+, \Sigma)$, to conclude that $\P({\mathcal N}_\delta)=0$ it suffices to show that 
	\begin{align*}
	\P\big({\mathcal S}(\zeta_r)=G,\,s\le r\le t \big)=0,
	\end{align*}
for any choice of times $s,\,t\in \Q$, $0<s<t<\delta$, and non $r$-absorbing set $G\in {\mathcal P}(V)$. Now 
	\begin{align*}
	&\P\big({\mathcal S}(\zeta_r)=G,\,s\le r\le t \big)
	=\lim_{a\downarrow 0}\lim_{b\downarrow 0}\P\Big(\{\inf_{s\le r\le t}\zeta_r(i)>a,\, i\in G\}\cap\big\{\inf_{s\le r\le t}\sum_{i\in G} \zeta_r(i)> 1-b \big\}\Big) \notag\\
	&\hspace{2cm}\le\lim_{a\downarrow 0}\lim_{b\downarrow 0}\liminf_{N\to \infty}P\Big(\{\inf_{s\le r\le t}\zeta^N_r(i)>a,\, i\in G\}\cap\big\{\inf_{s\le r\le t}\sum_{i\in G} \zeta^N_r(i)> 1-b \big\}\Big)\notag\\
	&\hspace{2cm}\le\lim_{a\downarrow 0}\lim_{b\downarrow 0}\liminf_{N\to \infty}P\Big(\{\inf_{s\le r\le t}\zeta^N_r(i)>a,\, i\in G\}\cap\big\{\sup_{j\in \mathcal A(G)\setminus G}\,\sup_{\frac{s+3t}{4} \le r\le t}\zeta^N_r(j)<b \big\}\Big)\notag\\
	&\hspace{2cm}=0
	\end{align*}
by \eqref{4.50} in Lemma \ref{now-eta}.
The second line above follows from the Portmanteau theorem and the remark  in \eqref{uocz}, as the set $\{\inf_{s\le r\le t}\zeta_r(i)>a,\, i\in G\}\cap\big\{\inf_{s\le r\le t}\sum_{i\in G} \zeta_r(i)> 1-b \big\}$ is open in $D(\R_+, \Sigma)$ with the local uniform topology. Then \eqref{2.910} holds.

Sicen $\tau_\delta\le \delta$, and $\delta$ is arbitrary, \eqref{2.910} implies that $\P\big(\{\inf_{t\ge 0},\,\zeta_t\in \Sigma_{abs}\}=0\big)=1$, i.e. $\P$ is supported on paths that satisfy condition {\bf(B)} in Definition~\ref{propab}. Proposition~\ref{conv} establishes that $(\zeta_t,\,t\ge 0)$ also satisfies condition {\bf (A)}, $\P$-a.s.. The result then follows from the uniqueness and identification of the solution to the $(\lambda, \mathcal D(V))$-problem, Proposition~\ref{promain}. \qed

\section{Perturbations to ${\mathcal D}_V$}\label{perturb}

We adapt the proof of the proof of Lemma 4.4 in \cite{BJL} to derive Lemma \ref{extens}.
 
 We will need the following result. 
	\begin{lemma}[Lemma 4.1 in \cite{BJL}]\label{BJL4.1} 
		Let $ D\subseteq V$ be nonempty.
		There exist a nonnegative function ${I_D:\Sigma\rightarrow \R}$
		 in $\mathcal{D}_D$ and constants $c$ and $C$ such that
	 		\begin{align}\label{a22}
			c \|\zeta\|_D^2 \leq I_D(\zeta) \leq C \|\zeta\|_D^2,
			\end{align}
		with $\|\zeta\|_D^2=\sum\limits_{j\in D}\zeta(j)^2$ and $0<c\le C<\infty$.
	\end{lemma}
\textit{Proof of Lemma \ref{extens}.}
Let $\phi\in C^\infty(\R,[0,1])$ such that $\phi(x)=0, x\le \frac13$, and $\phi(x)=1, x\ge 1$. Let $\lambda:=\frac{c}{3C}$, $c$ and $C$ the constants in \eqref{a22}, and 
$\tilde{\epsilon}:=\big(\frac{c\lambda^{|B|+1}}{2}\big)^{\frac12}\,\frac{\epsilon}{2}$. For $\emptyset\neq D\subseteq V$, define
	\begin{align}\label{a23}
	\Phi_D(\zeta)=\phi\Big(\frac{\lambda^{|D|} I_D(\zeta)}{\tilde{\epsilon}^2}-1\Big),
	\end{align}
$|D|$ the cardinality of $D$. By Lemma \eqref{BJL4.1}, it is simple to check that this function satisfies
	\begin{enumerate}
	\item[\it i)] $\Phi_D(\zeta)=0$ if $\|\zeta\|^2_D\le \frac43\frac{1}{C\lambda^{|D|}}\,\tilde{\epsilon}^2$,
	\vspace{1mm}
	\item[\it ii)] $\Phi_D(\zeta)=1$ if $\|\zeta\|^2_D\ge \frac{2}{c \lambda^{|D|}}\,\tilde{\epsilon}^2$,
	\vspace{1mm}
	\item[\it iii)] $\Phi_D(\zeta) \in {\mathcal D}_D$. 
	\end{enumerate}
Define 
	\begin{align}
	\Phi_k(\zeta)&=\prod_{D\subseteq B} \Phi_{D\cup\{k\}}(\zeta),\, k\in S,\quad  \Phi(\zeta)=\prod_{k\in S}\Phi_k(\zeta),\quad \text{and let}\quad f(\zeta)=h(\zeta)\Phi(\zeta). \label{a24}
	\end{align}
It follows from the choices of $\lambda$ and $\tilde{\epsilon}$ and property {\it ii)} above that $\zeta(k)\ge \frac12 \epsilon$ implies $\Phi_k(\zeta)= 1$, $k\in S$, and hence 
$\Phi\equiv 1$ in a neighbourhood of $\{u\in \Sigma,\,\min_{i\in A}u(i)\ge \epsilon\}$. This establishes \eqref{a21}.

Let us now verify that $f\in {\mathcal D}_V$. Let $k\in S$ and $\zeta \in \Sigma$ with $\zeta(k)=0$. If $\xi \in \Sigma$ is such that 
$\|\xi-\zeta\|^2<\frac43\frac{1}{C\lambda}\,\tilde{\epsilon}^2$ then in particular $|\xi_k|^2<\frac43\frac{1}{C\lambda}\,\tilde{\epsilon}^2$, and {\it i)} implies 
$\Phi_{\{k\}}(\xi)=\Phi_{\emptyset\cup\{k\}}=f(\xi)=0$. This shows that $f\equiv 0$ in a neighbourhood of $\zeta$, hence $\nabla f(\zeta)=0$ and the boundary condition 
$\langle \nabla f(\zeta),v_k\rangle=0$ is trivially satisfied, proving that $f\in {\mathcal D}_A$.

Let us now check that $f\in {\mathcal D} _B$. Given $j\in B$, we have
	\begin{align}\label{a25}
	\langle \nabla f(\zeta) ,  v_j\rangle=\Phi(\zeta) \,\langle \nabla h(\zeta), v_j  \rangle + h(\zeta)\,\sum_{k\in S} \sum_{D\subseteq B} \big( \hspace{-3mm}\prod_{\substack{k'\in S,\ D'\subseteq B\\ D'\cup\{k'\}\neq D\cup\{k\}}}\hspace{-4mm}\,\Phi_{D'\cup\{k'\}}(\zeta)\big)\, \langle \nabla \Phi_{D\cup\{k\}}(\zeta), v_j \rangle.\ \
	\end{align}
Let $\zeta \in \Sigma$ such that $\zeta(j)=0$. The first term above vanishes since by hypothesis $h\in {\mathcal D}_B$. Also, we claim that each term in the double sum on the right hand side vanishes as well. To see this, notice that if the set $D\subseteq B$ is such that $j\in D$, then $\Phi_{D\cup\{k\}}\in {\mathcal D}_{D\cup\{k\}}$ by {\it iii)}, and $\langle \nabla \Phi_{D\cup\{k\}}(\zeta), v_j\rangle=0$. If, on the other hand, $j\notin D$, then $\Phi_{D\cup\{j,\,k\}}$ is one of the factors in the product in front of the brackets, 
	\begin{align}\label{a25bis} 
	(\prod_{\substack{D'\subseteq B \\ D'\neq D}}\hspace{-2mm}\,\Phi_{D'\cup\{k\}}(\zeta))\, \langle \nabla \Phi_{D\cup\{k\}}(\zeta), v_j \rangle=
	(\hspace{-3mm}\prod_{\substack{ D'\neq D \\D'\neq D\cup \{j\}}}\hspace{-4mm}\,\Phi_{D'\cup\{k\}}(\zeta))\, \Phi_{D\cup\{j,\,k\}}(\zeta)\,\langle \nabla \Phi_{D\cup\{k\}}(\zeta), v_j \rangle.
	\end{align}
Now, by {\it i)} and {\it ii)},
	\begin{align}
	\Phi_{D\cup\{j,\,k\}}(\xi)&=0\,\text{ if }\,\|\xi\|^2_{D\cup\{j,\,k\}}\le \textstyle{\frac43 \frac{1}{C\lambda^{|D|+2}}\tilde{\epsilon}^2},\label{a26}\\
	\Phi_{D\cup\{k\}}(\xi)&=1\,\text{ if }\,\|\xi\|^2_{D\cup\{k\}}\ge \textstyle{\frac{2}{c\lambda^{|D|+1}} \tilde{\epsilon}^2} \quad\text{and}\quad \nabla\Phi_{D\cup\{k\}}\equiv 0 \text{ on } \|\xi\|^2_{D\cup\{k\}}> \textstyle{\frac{2}{c\lambda^{|D|+1}}\tilde{\epsilon}^2}. \label{a27}
	\end{align}
With the choice of $\lambda=\frac{c}{3C}$ we have $ \frac43 \frac{1}{C\lambda^{|D|+2}}>\frac{2}{c\lambda^{|D|+1}}$. If $\zeta\in \Sigma$ is such that $\zeta(j)=0$ then
$\|\zeta\|_{D\cup\{j,\,k\}}=\|\zeta\|_{D\cup\{k\}}$, either \eqref{a26} or \eqref{a27} holds, and in both cases the expression in \eqref{a25bis} vanishes. This completes the proof of our claim, and it follows from \eqref{a25} that $\langle \nabla f(\zeta) ,  v_j\rangle=0$ when $\zeta(j)=0$. \qed

\section*{Acknowledgements}\label{ackn}
We are grateful to Matthieu Jonckheere and Daniel Valesin for very helpful discussions.

I.A.'s research was supported by grants PICT2015-3583 and PIP11220130100521CO.

\bibliographystyle{plain}

\end{document}